\documentclass[9pt]{article}
\usepackage{amsmath,amsthm,amsfonts,amssymb,amscd,url}
\usepackage[backref=page,colorlinks=false]{hyperref}
\usepackage[capitalize]{cleveref}
\usepackage[english]{babel}
\usepackage{graphics}
\numberwithin{equation}{section}
\input{xy} \xyoption{all}
\usepackage[latin1]{inputenc}
\usepackage{svg}

\usepackage{enumerate,xcolor}
\usepackage{hyperref, mathrsfs}
\usepackage{epsfig}
\input{xy} \xyoption{all}
\definecolor{orange}{RGB}{0,0,0}
\definecolor{blue}{RGB}{0,0,0}
\usepackage[T1]{fontenc}

\usepackage[scr=boondoxo,scrscaled=1.05]{mathalfa}
\makeindex

\def\Card{\mathrm{Card\, }}

\def\exp{\mathrm{exp}}

\def\loc{\mathrm{loc}}

\def\R{\mathbb R}
\def\N{\mathbb N}
\def\Z{\mathbb Z}

\def\D{\mathbb D}

\def\C{\mathbb C}

\def\loc{\mathrm{loc}}

\def\cl{\mathrm{cl}}

\def\supp{\mathrm{supp}\:}

\def\qand{\quad \text{and} \quad}

\def\1{{1\!\! 1}}
\theoremstyle{plain}

\newtheorem{thm}{\bf Theorem}[section]
\newtheorem{theorem}[thm]{\bf Theorem}

\newtheorem*{conjecture*}{\bf Conjecture}

\newtheorem{problem}[thm]{\bf Problem}

\newtheorem{proposition}[thm]{\bf Proposition}
\newtheorem{corollary}[thm]{\bf Corollary}
\newtheorem{lemma}[thm]{\bf Lemma}

\newtheorem*{Takens prbm}{Takens' Last Problem} 
\newtheorem{remark}[thm]{\bf Remark}
\newtheorem{fact}[thm]{\bf Fact}

\newtheorem{definition}[thm]{\bf Definition}
\newtheorem*{definition*}{\bf Definition}
\newtheorem*{question*}{\bf Question}
\newtheorem{theo}{\bf Theorem}

\newtheorem{example}[thm]{\bf Example}

\renewcommand*{\backref}[1]{}
\renewcommand*{\backrefalt}[4]{\quad \tiny 
  \ifcase #1 (\textbf{NOT CITED.})%
  \or    (Cited on page~#2.)%
  \else   (Cited on pages~#2.)%
  \fi}
\usepackage{libertine}
\newcommand{\supess}{ \mathrm{ess \ sup\,  }}
\newcommand{\infess}{ \mathrm{ess \ inf\,  }}

\def \Scl{\mathsf{scl}}
\def \scl{\mathrm{scl}}

\def \Card{\mathrm{Card \, }}

\def \dim{\mathsf{dim}}
\def \ord{\mathsf{ord}}

\def \R{\mathbb{R}}
\def \N{\mathbb{N}}
\begin{document}
\selectlanguage{english} 
 \title{Scales}
\author{Mathieu Helfter \footnote{Partially supported by the ERC project 818737 \textit{Emergence on wild differentiable dynamical systems}} }
\date{\today}
\maketitle
\begin{abstract}
We introduce the notions of \emph{scale} for sets and measures on metric space
by generalizing the usual notions of dimension. Several versions of scales are introduced such as \emph{Hausdorff, packing, box, local and quantization}. They are defined for different \emph{growth}, allowing a refined study of infinite dimensional spaces. We prove general theorems comparing the different versions of scales. They are applied to describe geometries of ergodic decompositions, of the Wiener measure and from functional spaces. The first application solves a problem of Berger on the notions of emergence (2020); the second lies in the geometry of the Wiener measure and extends the work of Dereich-Lifshits (2005); the last refines Kolmogorov-Tikhomirov (1958) study on finitely differentiable functions. 
\end{abstract}
\tableofcontents
\section{Introduction and results} 
 Dimension theory was popularized by Mandelbrot in the article \emph{How long is the coast of Britain ?} \cite{mandelbrot1967long} and shed light on the general problem of measuring how large a natural object is.
 The category of objects considered are metric spaces possibly endowed with a measure. Dimension theory encompasses not only smooth spaces such as manifolds, but also wild spaces such as fractals, so that the dimension may be any non-negative real number. There are several notions of dimension: for instance Hausdorff \cite{hausdorff1918dimension}, packing \cite{tricot1982two} or box dimensions \cite{bouligand1928ensembles}. Also, when the space is endowed with a measure, there are moreover the local and the quantization dimensions. These different versions of dimension are bi-Lipschitz invariants. They are in general not equal, so that they reveal different aspects of the underlying space. Seminal works by Hausdorff, Frostman, Tricot, Fan, Tamashiro, P{\"o}tzelberger, Graf-Luschgy, and Dereich-Lifshits have described the relationships between these notions and provided conditions under which they coincide.

Obviously, these invariants do not give much information on infinite dimensional spaces. However, such spaces are the subject of many studies. As motivations, Kolmogorov-Tikhomirov in  \cite{tikhomirov1993varepsilon}  gave asymptotics of  the covering numbers of some functional spaces. Dereich-Lifshits gave asymptotics of the mass of the small balls for the Wiener measure and exhibited their relationship with the quantization problem, see \cite{dereich2003link,dereichproba05,Cameron1944TheWM,chung1947maximum,baldi1992some,kuelbs1993metric}. Also Berger and Bochi \cite{berger2020complexities} gave estimates on the covering number and quantization number of the ergodic decomposition of some smooth dynamical systems. See also \cite{baldi1992some,kloeckner2015geometric,berger2021emergence}. 

These results lead to the following:

\begin{question*}
Are there infinite-dimensional counterparts of the various versions of dimension that maintain similar relationships ?
\end{question*}
To address this question, we introduce  
the notion of \emph{scale}.
The key idea involves considering a \emph{scaling}, which is a one-parameter family of gauge functions that satisfies mild assumptions, dictating at which 'scale' the size of the space is examined. For instance the families for the $\dim$ for the  \emph{dimension} and $ \ord$ for the \emph{order} given in \cref{example of growth} are scalings. Given a scaling, different versions of scales are defined. In particular, Hausdorff dimension, packing dimension or box dimension are scales.

Given a scaling, we will generalize comparison theorems between the different kinds of dimensions to all the different kinds of scales in \cref{resultsA}, \ref{resultsB} and \ref{resultsC}. The definition of scaling is crafted so that the proofs for \cref{resultsA} and \cref{resultsB} are nearly direct extensions of established results from dimension theory (see \cref{sec compp}).

\medskip
\emph{The main difficulty will be then to prove \cref{resultsC} which enables us to compare the quantization scales with  both the local and the box scales. 
Also even for the specific case of dimension, new inequalities between quantization dimension of a measure and box dimension of the set of positive mass are proven in \cref{resultsC} (inequalities (f) and (h)). Main novelties from \cref{resultsC} are reformulated in \cref{lem quant boite rep au pb,rep au pb}.}
\medskip

In the next \cref{lieu from dim to scale} we recall usual definitions of dimension and introduce the notions of scaling and scales. The theorems comparing the different versions of scales are stated in \cref{sec compp}.
Precise definitions of the involved scales are given in \cref{lieu metric scl}, and in \cref{lieu meas scl} for measures. 
Then, \cref{applicationscales} is devoted to applications of the main results. 
In \cref{lieu wiener}, a first  application of \cref{resultsC} together with Dereich-Lifshits estimate in \cite{dereichproba05} implies the coincidence of local, Hausdorff, packing, quantization and box orders of the Wiener measure for the $L^p$-norm, for any $ p \in [1, \infty]$. Then in \cref{app func}, we apply \cref{resultsA} to show the coincidence of the box, Hausdorff and packing orders for finitely regular functional spaces; refining the Kolmogorov-Tikhomirov study in  \cite[Thm XV]{tikhomirov1993varepsilon}. Lastly in \cref{app emergence}, a consequence of \cref{resultsC} is that the local order of the ergodic decomposition is at most its quantization order. This solves a problem set by Berger in \cite{berger2020complexities}.

\medskip

\paragraph{Thanks:}

\textit{Warmest thanks to Pierre Berger for his dedication and advice; to Ai-Hua Fan for his interest, references, and guidance; to Fran\c{c}ois Ledrappier for answering my questions and providing references; to Martin Leguil for his interest; and to Camille Tardif and Nicolas de Saxc\'e for their helpful references. I would also like to particularly thank the referee for their insightful and highly detailed comments, which helped considerably to improve the clarity of the proofs and the overall structure of this article.}

 \medskip

\subsection{From dimension to scale} \label{lieu from dim to scale}
Let us first recall some classical definitions from dimension theory and see how they could be naturally extended to define finite invariants for infinite dimensional spaces. The Hausdorff, packing and box dimensions of a totally bounded metric space $(X,d)$ are defined by looking at families of subsets of $X$.
First consider the box dimension. Recall that, given an error $ \epsilon > 0$, the covering number $ \mathcal{N}_\epsilon (X) $ is the minimal cardinality of a covering of $X$ by balls of radius $ \epsilon$. Then the \emph{lower and upper box dimensions} of $(X,d)$ are given by: 
\[ \underline{\dim}_B X := \sup \lbrace \alpha > 0 :  \mathcal{N}_\epsilon (X)  \cdot \phi_\alpha (\epsilon) \to + \infty \rbrace \qand  \overline{\dim}_B X := \inf \lbrace \alpha > 0 : \mathcal{N}_\epsilon (X)  \cdot \phi_\alpha (\epsilon) \to 0  \rbrace \;  ,   \]
where $ (\phi_\alpha)_{\alpha >0} $ is the family of functions on $(0,1)$ given for $ \alpha > 0 $ by $\phi_\alpha : \epsilon  \mapsto \epsilon^\alpha $.

Box dimensions, also called box counting dimensions or Minkovski-Bouligand dimensions were introduced by Bouligand in \cite{bouligand1928ensembles}. 
In general, upper and lower box dimensions do not coincide (see e.g. \cite{falconer1997techniques,falconer2004fractal}). However, when $ X$ is a smooth manifold endowed with the Euclidean metric, these two dimensions coincide with the usual definition of dimension. 
Basic properties of  box dimensions are revealed when looking at subsets of a metric space with the induced metric. Notably, box dimensions are non decreasing for the inclusion and are invariant by topological closure.  
In general they are not \emph{$\sigma$-stable}, i.e. the box dimensions of a countable union of subsets of a metric space are a priori not equal to the suprema of the corresponding dimensions of the subsets. 
The most popular version  of dimension that enjoys the property of $\sigma$-stability is  \emph{Hausdorff dimension}. Let us recall its definition. Given an error $ \epsilon > 0 $, consider:
 \[\mathcal{H}_\epsilon^\alpha (X)    := \inf_{  C \in \mathcal{C}_H (\epsilon)  } \sum_{ B(x , \delta )  \in C  }  \phi_\alpha ( \delta )  \; ,   \]
where $ \mathcal{C}_H (\epsilon) $ is the set of countable coverings of $ X$ by balls of radius at most $ \epsilon$. 
Then, the \emph{Hausdorff dimension} of $(X,d)$ is given by: 
\[ \dim_H   X =  \sup \left\{ \alpha > 0 :  
\mathcal{H}_\epsilon^\alpha (X)   \xrightarrow[\epsilon \rightarrow 0]{}   + \infty  \right\}  =   \inf \left\{ \alpha > 0 : \mathcal{H}_\epsilon^\alpha (X)   \xrightarrow[\epsilon \rightarrow 0]{}   0  \right\}  \; .  \]
Lastly, another interesting dimension that enjoys $\sigma$-stability is the \emph{packing dimension}. Its construction is analogous to that of Hausdorff dimension and was introduced by Tricot in his thesis \cite{tricot1982two}. It is actually linked to the upper box dimension by the following characterization that we will use for the moment as a definition:
\[ \dim_P X :=  \inf \sup_{n \ge 1 } \overline{\dim}_B E_n  \; ,  \]
where the infimum is taken over countable coverings $ ( E_n )_{n \ge 1} $ by subsets of $X$. 
These four versions of dimension are bi-Lipschitz invariants; they quantify different aspects of the geometry of the studied metric space since they a priori do not coincide. However, it always holds:
\[ \dim_H X  \le \underline{\dim}_B X \le \overline{\dim}_B X  \qand  \dim_H X  \le \dim_P  X \le \overline{\dim}_B X  \; .  \]
See e.g. \cite{falconer1997techniques}, \cite{falconer2004fractal} for detailed proofs.

Let us now introduce \emph{scales}. A simple observation is that all of the above versions of dimension involve a specific parameterized family $ ( \phi_\alpha)_{\alpha > 0} = ( \epsilon \mapsto \epsilon^\alpha )_{ \alpha > 0 } $ of \emph{gauge functions} with polynomial behavior. In dimension theory, gauge functions generalize the measurement of ball diameters, enabling more control over the definition of the Hausdorff measure in the finite dimensional case. For scales, the objective is different. Roughly speaking, we will allow gauge functions to exhibit behaviors that are far from being polynomial. 
Let us precise the discussion. If a space $ (X,d) $ is infinite dimensional then its covering number $ \mathcal{N}_\epsilon(X) $ grows faster than any polynomial in $ \epsilon^{-1} $ as $ \epsilon$ decreases to $ 0$. In order to define finite invariants for infinite dimensional spaces we must allow other gauge functions that decrease faster than any polynomial when the radius of the involved balls decreases to $ 0$. 
Consequently, we propose to replace the family $(\phi_\alpha)_{\alpha>0}= (\epsilon \in ( 0, 1 ) \mapsto \epsilon^\alpha)_{\alpha > 0 } $  in all the above definitions of dimensions, by other families of gauge functions that encompass the following examples of growth: 
\begin{example}\label{example of growth}
\begin{enumerate}
\item The family   $\dim = ( \epsilon	 \in ( 0,1) \mapsto \epsilon^\alpha ) _{ \alpha > 0 }$ which is used in the definitions of dimensions, 
\item the family $\ord = ( \epsilon \in ( 0,1 ) \mapsto  \exp ( - \epsilon^{-\alpha} )) _{\alpha> 0 }$ which is called \emph{order}. It fits with the growth of  the covering number of spaces of finitely regular functions studied by Kolmogorov-Tikhomirov  \cite{tikhomirov1993varepsilon}, see \cref{KT},  or with the one of the space of ergodic measures spaces by Berger-Bochi \cite{berger2020complexities}, as we will see in \cref{ans Berger}, 
\item the family $  \left(  \epsilon \in ( 0,1 ) \mapsto \exp ( - (\log \epsilon^{-1} ) ^\alpha )\right) _{\alpha > 0 }$ which fits with the growth  of  the covering number of holomorphic functions estimated by Kolmogorov-Tikhomirov  \cite{tikhomirov1993varepsilon}, as we will see in \cref{holo}.  
\end{enumerate}
\end{example}
To properly extend definitions and comparison theorems among different scales, i.e. the generalized box, Hausdorff, and packing dimensions; the family of functions \((\phi_\alpha)_{\alpha > 0}\) must satisfy certain mild assumptions. This requirement leads us to introduce the notion of \emph{scaling}: 
\begin{definition}[Scaling] \label{scaling}
A family  $ \Scl = (\scl_\alpha)_{\alpha > 0} $  of positive non-decreasing functions on  $(0, 1)$ is  a \emph{scaling} when for every $ \alpha> \beta > 0 $ and any $\lambda > 1 $ close enough to $1$, it holds:  
  \begin{equation}  \tag{$ * $} \label{def scl}  \scl_\alpha  ( \epsilon)  = o \left(  \scl_\beta ( \epsilon^\lambda ) \right)   \qand   \scl_\alpha  ( \epsilon) = o  \left(   \scl_\beta ( \epsilon ) ^\lambda  \right)   \, \text{ as }  \epsilon \rightarrow 0  \; . \end{equation}
\end{definition} 
\begin{remark}
The left hand side condition is used in all the proofs of the theorems represented on \cref{fig:my_label}. The right hand side condition is only used to prove the equalities between packing and upper local  scales in \cref{resultsB} and to compare upper local scales with upper box and upper quantization scales in \cref{resultsC} inequalities $(c) \& (g)  $. It also allows us to characterize packing scale with packing measure. 
\end{remark}

\begin{remark}
There are scalings that allow one to study $0$-dimensional spaces, for instance the family: \[ \left(  \epsilon \mapsto  \log ( \epsilon^{-1} \right)^{ - \alpha} ) _{ \alpha > 0 }   \; . \] 
\end{remark}

 We will show in  \cref{ex scl} that the families in \cref{example of growth} are scalings. 
Scalings enable defining  \emph{scales} that generalize packing dimension, Hausdorff dimension,  box dimensions,  quantization  dimensions and local dimensions.
For each scaling, the different kinds of scales do not a priori coincide on a generic space. 
Nevertheless, in \cref{applicationscales}, as a direct application of our comparison  theorems, we bring examples of metric spaces and measures where all those definitions coincide. In these examples,  equalities between the different scales are linked to some underlying "homogeneity" of the space that is provided by the existence of an equilibrium state. 

Now for a metric space $ (X,d)$, replacing the specific family $  \dim $ in the definition of box dimensions by a given scaling $ \Scl = ( \scl_\alpha ) _{ \alpha > 0} $ gives the following: 
\begin{definition}[Box scales] \label{def box}
\emph{Lower and upper  box scales} of a metric space $ (X,d) $ are defined by:  
\[\underline \Scl_B X=  \sup \left\{ \alpha > 0 :  \mathcal{N}_\epsilon (X) \cdot \scl_\alpha ( \epsilon )\xrightarrow[\epsilon \rightarrow 0]{}   + \infty  \right\} \]
and
\[ \overline{\Scl}_B   X =   \inf \left\{ \alpha > 0 : \mathcal{N}_\epsilon (X)  \cdot  \scl_\alpha ( \epsilon ) \xrightarrow[\epsilon \rightarrow 0]{}  0 \right\}  \;.\] 
\end{definition}
Moreover, we will generalize the notion of Hausdorff and packing dimensions to the \emph{Hausdorff scale} denoted $ \Scl_H X $  (see \cref{H scale}) and \emph{packing scale}  denoted $ \Scl_P X$ (see \cref{packing l box}). The constructions are fully detailed in the next section. Let us now state the main results on the comparison of scales of metric spaces.
\subsection{Results on the comparison of scales \label{sec compp}}
\begin{figure}[!ht]
    \centering
    \includegraphics[scale=.15]{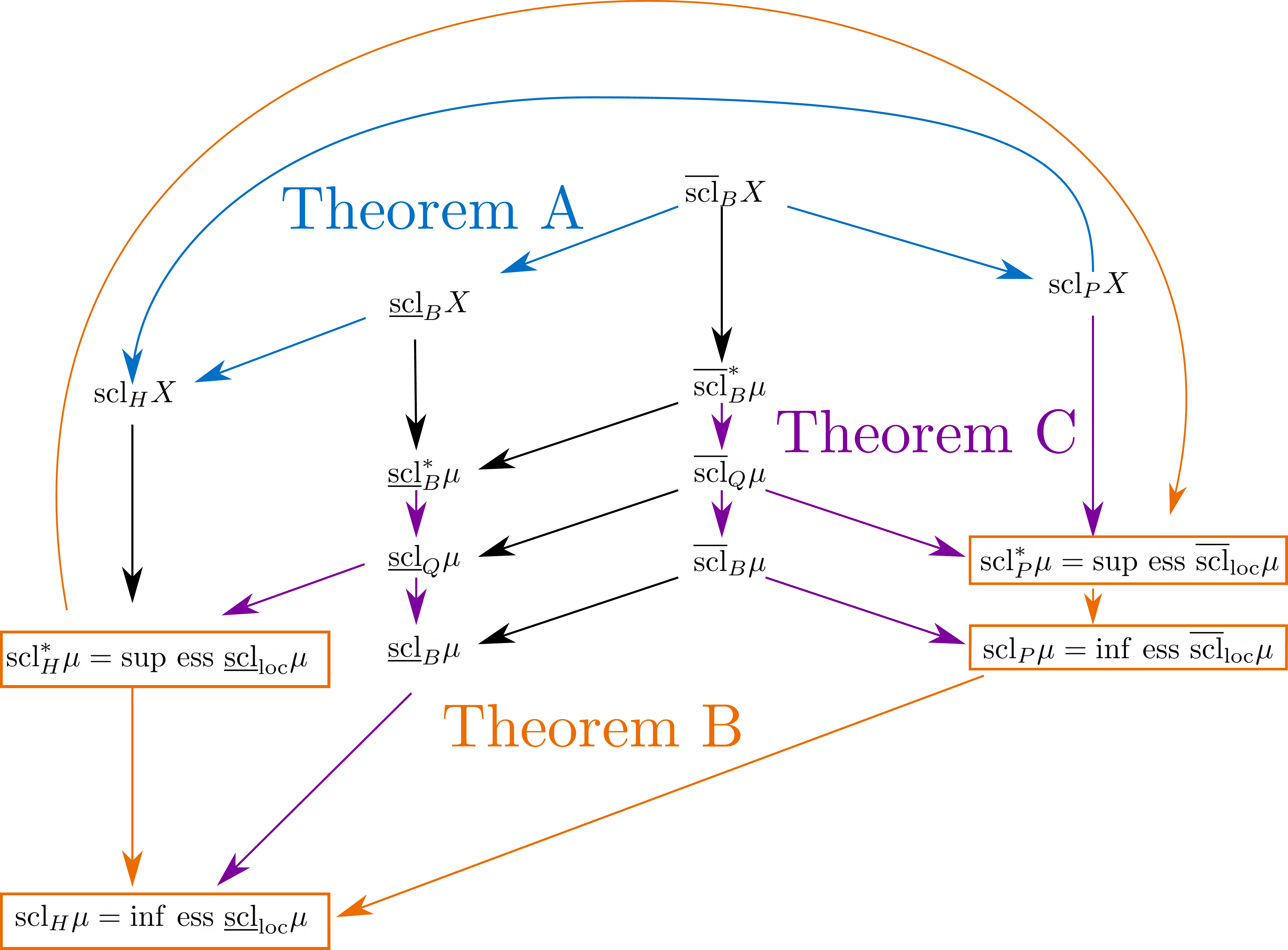}
    \caption{Diagram presenting results of Theorems \ref{resultsA}, \ref{resultsB} and \ref{resultsC}. }
    {Each arrow is an inequality, the scale at the starting point of the arrow is at least the one at its ending point : $" \rightarrow " = " \ge" $.
    } 
    \label{fig:my_label}
\end{figure}
In this section, we introduce other kinds of scales and Theorems  \ref{resultsA}, \ref{resultsB} and \ref{resultsC}  that state the inequalities between them as illustrated in \cref{fig:my_label}.

First, we bring the following generalization of classical inequalities comparing dimensions of metric spaces to the frame of scales: 

\medskip
\begin{theo}
\label{resultsA}
Let $( X,d)$ be a metric space and $ \Scl$ a scaling, the following inequalities hold: 

\[ \Scl_H X \le \Scl_P X  \le \overline \Scl_B X \qand\Scl_H X  \le \underline  \Scl_B X  \le  \overline \Scl_B X \; . \]
\end{theo} 

 In the specific case of dimension, these inequalities are well known and presented for instance by Tricot \cite{tricot1982two} or Falconer \cite{falconer1997techniques,falconer2004fractal}. 
The proof of this theorem will be done in \cref{comp of scales}.  The key part is to show that Hausdorff scales and packing scales are well defined quantities. Then we will follow the lines of Falconer's proof to show \cref{resultsA}.
\medskip 

The relationship between the dimensions of \((X, d)\) and its measures was first studied by Frostman \cite{frostman1935potentiel}, who used equilibrium states to describe the Hausdorff dimension of the underlying space. Conversely, the dimensions of sets can characterize the asymptotic behavior of the mass of small balls of a measure. This concept was introduced by Fan as \emph{local dimension} \cite{fan1994dimensions}, leading to seminal studies by Fan, Lau, and Rao \cite{fan2002relationships}, P\"otzelberger \cite{potzelberger1999quantization}, and Tamashiro \cite{tamashiro1995dimensions}.
Similarly we introduce \emph{local scales} that extend the notion of local dimensions of a measure:
\begin{definition}[Local scales] \label{def loc}
Let $ \mu$ be a Borel measure on a metric space $(X,d)$ and $ \Scl $ a scaling. The \emph{lower and upper scales} of $\mu$ are the functions that map a point $x \in X$ to:
\[\underline{\Scl}_{\loc}  \mu (x)  = \sup \left\{  \alpha > 0 : \frac{\mu \left(   B(x,  \epsilon ) \right) }{ \scl_\alpha  ( \epsilon )}   \xrightarrow[\epsilon \rightarrow 0]{}  0\right \} \] and \[\overline{\Scl}_{\loc}  \mu (x) = \inf \left\{ \alpha > 0 :  
\frac{\mu \left(   B(x,  \epsilon ) \right) }{ \scl_\alpha  ( \epsilon )}   \xrightarrow[\epsilon \rightarrow 0]{}  +  \infty
\right\} \;.\]
\end{definition} 
As in dimension theory, we should not compare the local scales with the scales of $X$ but to those of its subsets with positive mass. This observation leads to considering the following: 
\begin{definition}[Hausdorff, packing and box scales of a measure] \label{metric to measure}
Let $ \Scl$ be a scaling and $ \mu$ a non-null Borel measure on  a metric space $(X,d)$. 
For any $ \Scl_\bullet \in \left\{ \Scl_H, \Scl_P, \underline \Scl_B, \overline \Scl_B \right\}$  we define 
lower and upper scales of the measure $ \mu$ by: 
\[\Scl_\bullet \mu = \inf_{ E \in \mathcal{B} } \left\{ \Scl_\bullet E : \mu(  E ) >  0 \right\} \qand  \Scl^*_\bullet\mu = \inf_{ E \in \mathcal{B} }  \left\{ \Scl_\bullet E : \mu( X \backslash E ) = 0 \right\} \;,  \]
where $\mathcal{B}$ is the set of Borel subsets of $X$. 
\end{definition}
In the case of dimension  Fan  \cite{fan1994dimensions,fan2002relationships} and Tamashiro \cite{tamashiro1995dimensions} exhibited the relationships between the Hausdorff and packing dimensions of measures and their local dimensions that we generalize as: 
\begin{theo}
\label{resultsB}
Let $\mu$ be a Borel measure on a  metric space $(X,d)$, then for any scaling $ \Scl$, Hausdorff and packing scales of $\mu$ are characterized by:   
\[  \Scl_H \mu = \infess  \underline \Scl_{\loc} \mu, \quad \Scl^* _H \mu = \supess  \underline \Scl_{\loc} \mu, \quad
  \Scl_P \mu = \infess  \overline \Scl_{\loc} \mu , \quad \Scl^* _P \mu = \supess  \overline \Scl_{\loc} \mu \;,  \]
  where
 $ \supess$ and $ \infess$  denote the essential suprema and  infima of a function.
\end{theo} 
The proof of \cref{resultsB} is provided in \cref{proof thm B}. It is inspired by the proofs of Fan \cite{fan1994dimensions} and Tamashiro \cite{tamashiro1995dimensions} in the dimensional case.  

Let us introduce a last kind of scale, the \emph{quantization scale}. It generalizes the quantization dimension. This definition is motivated by the following works \cite{graf2007foundations, potzelberger1999quantization,dereich2003link,dereichproba05,berger2017emergence,berger2021emergence,berger2020complexities}. 
\begin{definition}[Quantization scales] \label{def quant}
Let $(X, d)$ be a metric space and  $\mu$ a Borel measure on $X$.
Given an error $ \epsilon >0 $, the \emph{quantization number} $\mathcal Q_\mu( \epsilon) $ of $ \mu$ is the minimal cardinality of a finite set of points that is  $ \epsilon$-close on average to any point in $X$:
\[ \mathcal{Q}_\mu(\epsilon) = \inf  \left\{ N \ge 0 : \exists 
\left\{ c_i \right\}_{i= 1, \dots, N } \subset X,  \int_X d(x,  \left\{ c_i\right\}_{  1 \le i \le N } ) d \mu (x) \le \epsilon  \right\} \; . \]
Then \emph{lower and upper quantization scales} of $ \mu$  for a given scaling $ \Scl$ are defined by:
\[ \underline{\Scl}_Q \mu  =  \sup \left\{ \alpha > 0 :   \scl_\alpha ( \epsilon ) \cdot   \mathcal{Q}_\mu ( \epsilon )  \xrightarrow[\epsilon \rightarrow 0]{}  + \infty  \right\}   \] and \[
\overline{\Scl}_Q \mu  =  \inf \left\{ \alpha > 0 :    \scl_\alpha ( \epsilon ) \cdot  \mathcal{Q}_\mu ( \epsilon )   \xrightarrow[\epsilon \rightarrow 0]{}  0  \right\}  \; . \]
\end{definition}
The following gives relationships between the  different kinds of scales of measures:
\begin{theo}[Main]
\label{resultsC}
Let $(X, d)$ be a metric space. 
Let  $\mu$ be a Borel measure on $X$. For any scaling $ \Scl$, the following inequalities on the scales of $\mu$ hold:  
\[  \infess \underline{\Scl}_{\loc} \mu  \underbrace{\le}_{(a)}  \underline   \Scl_B \mu  \underbrace{\le}_{(b)} \underline \Scl_Q \mu  \quad ; \quad \infess \overline{\Scl}_{\loc} \mu   \underbrace{\le}_{(c)}  \overline   \Scl_B \mu  \underbrace{\le}_{(d)}\overline \Scl_Q \mu   \]
and 
\[  \supess \underline{\Scl}_{\loc} \mu  \underbrace{\le}_{(e)}\underline \Scl_Q \mu   \underbrace{\le}_{(f)}  \underline   \Scl^*_B \mu   \quad ; \quad \supess \overline{\Scl}_{\loc} \mu  \underbrace{\le}_{(g)}\overline \Scl_Q \mu  \underbrace{\le}_{(h)}  \overline   \Scl^*_B \mu\;  .  \]
\end{theo}
Inequalities $(b)$ and $(d)$  are part of \cref{lem quant boite rep au pb} and rely mainly on the use of the Borel-Cantelli lemma. Even in the specific case of dimension, these inequalities are new, as far as we know.  
Inequalities $(e)$ and $(g)$  were shown by  P\"otzelberger in  \cite{potzelberger1999quantization} for dimension and in $[ 0, 1]^d$. A new approach for the general case is brought in \cref{rep au pb}. We deduce the inequality $(a)$  from $ (e) $ and $(f)$ and the inequality $(c)$ from $(g)$ and $(h)$. The proof of inequalities $(f)$ and $(h)$ is straightforward, see \cref{quantisation boite}.

As a direct application,  inequality $ (e)$ allows us to answer a problem set by Berger in \cite{berger2020complexities} (see \cref{app emergence}). 
We will give in \cref{scales of group} examples of spaces for which the different versions of orders do not coincide.
\medskip
\subsection{Applications \label{applicationscales}} 

Let us see how our main theorems easily imply the equality between the different scales of some natural infinite dimensional spaces.

\subsubsection{Wiener measure} \label{lieu wiener}
The first example is the computation of the orders of the Wiener measure $W$ that describes uni-dimensional standard Brownian motion on $ [0,1 ] $. Recall that $W$ is the law of a continuous process $ ( B_t)_{ t \in [0,1] }$ with independent increments. It is such that for any $ t \ge s$  the law of the random variable $ B_t - B_s $ is $ \mathcal{N} ( 0, t-s )  $. 
Computation of the local scales of the Wiener measure relies on small ball estimates which received much interest \cite{Cameron1944TheWM,chung1947maximum, baldi1992some, kuelbs1993metric}. These results gave asymptotics on the measure of small balls centered at $0$ for  $L^p$ norms  and H\"older norms. 
Moreover, for a random ball, Dereich-Lifshits made the following estimate for $L^p$-norms:
\begin{theorem}[Dereich-Lifshits {\cite{dereichproba05}[3.2, 5.1, 6.1, 6.3]}] \label{ord loc Wiener}
For the Wiener measure on $ C^0([ 0,1 ] , \R ) $  endowed with the $L^p$-norm, for $ p \in [1, \infty]$, there exists \footnote{ Note that for $ p <  \infty$, the constant $ \kappa$ does not depend on the value of $p$. } a constant $ \kappa > 0$ such that for  $W$-almost any $ \omega  \in C^0 ( [0,1] , \R ) $: 
\[  - \epsilon^2 \cdot \log   W ( B(\omega, \epsilon ))  \rightarrow \kappa, \text{when $\epsilon \rightarrow 0$ } \; ,   \]
and moreover the quantization number of $W$ verifies: 

\[ \epsilon^2 \cdot \log   \mathcal{Q} _W ( \epsilon )  \rightarrow \kappa, \text{when $\epsilon \rightarrow 0$ } \;  .    \]
\end{theorem}

As a direct consequence of \cref{resultsB} and \cref{resultsC}  we get that the new invariants we introduced for a measure with growth given by $ \ord$ all coincide: 
\begin{theo}[Orders of the Wiener measure] \label{brownian}
The Wiener measure on $ C^0([ 0,1 ] , \mathbb R ) $  endowed with the $L^p$-norm, for $ p \in [1, \infty]$, verifies for $W$ almost every $\omega \in C^0 ( [0,1]) $ : 
\begin{align*}
   2 &=  \underline \ord_{\loc} ( \omega )  = \ord_H W = \ord^*_H W =  \underline \ord_B W = \underline \ord_Q  W\\
   &= \overline \ord_{\loc} ( \omega )  =  \ord_P W = \ord^*_P W = \overline \ord_B W =  \overline \ord_Q W \; .
\end{align*}
\end{theo} 
In particular, the framework of scales allows us to define Hausdorff, packing and box orders which are equal to $2$ for the Wiener measure. This indicates some kind of constant geometric property of subsets of maps with positive Wiener measure. 
\begin{proof}
By \cref{ord loc Wiener}, for $W$-almost $ \omega$ and for any $ p \in [1, \infty]$, in the $L^p$-norm it holds: 
\[   \underline \ord_{\loc} W(\omega) = \overline \ord_{\loc} W ( \omega) = 2  =  \underline \ord_{Q} W= \overline \ord_{Q} W \; . \] 
Now by \cref{resultsB}, it holds:
  \[     \ord_H W  =  \underline \ord_{\loc} W(\omega) =  \ord^*_H W   \qand  \ord_P W = \overline \ord_{\loc} W (\omega) =  \ord^*_P W   \; . \] 
Finally, since by \cref{resultsC} we have: 
\[  \overline \ord_Q W \ge \overline \ord_B W \ge \underline \ord_B W \ge \ord_H W \;, \]
the desired result comes by combining the three above lines of equalities and inequalities. 		
\end{proof}

\
\begin{remark}
Since $ (C^0( [0,1 ] , \R ) , L^p ) $  is not (totally) bounded, as well as any of its subsets with total mass, it holds $ \underline  \ord^*_B = + \infty $.  
\end{remark}

\subsubsection{Functional spaces endowed with the $C^0$-norm \label{app func}}
Let $d $ be a positive integer. For any integer $ k  \ge 0 $ and for any $\alpha \in (0,1]$ denote:  
\[ \mathcal F^{d,k,0} :=  \left\{ f \in C^k ( [0,1]^d, [-1,1]) : \| f \|_{C^k} \le 1 \right\} \]
and\[ \mathcal F^{d,k,\alpha} :=  \left\{ f \in C^k ( [0,1]^d, [-1,1]) : \| f \|_{C^k} \le 1 \text{ and $D^k f$ is $\alpha$-H{\"o}lder with constant $1$ } \right\} \; .  \]
We endow these spaces with the $C^0$ uniform norm (see \cref{proof func} for details). 

Kolmogorov and Tikhomirov gave the following asymptotics:

\begin{theorem}[Kolmogorov-Tikhomirov, {\cite{tikhomirov1993varepsilon}[Thm XV]}] \label{KT} 
Let $d $ be a positive integer. For any integer $ k $ and for any $\alpha \in [0,1]$, there exist two constants $ C_1 > C_2 > 0$ such that for every $ \epsilon > 0$,  the covering number $ \mathcal{N}_{\epsilon} ( \mathcal{F}^{d, k, \alpha} ) $ of the space $(  \mathcal{F}^{d,k,\alpha}, \| \cdot \|_\infty ) $ verifies:  
\[ C_1  \cdot \epsilon^{-\frac{d}{ k + \alpha } }  \ge   \log \mathcal{N}_\epsilon ( \mathcal{F}^{d, k, \alpha} )  \ge  C_2  \cdot  \epsilon^{-\frac{d}{ k + \alpha } } \; .   \]
\end{theorem}
In \cref{proof of lem diff}, we will embed an infinite-dimensional Kantor set into \(\mathcal{F}^{k,d,\alpha}\) via an expanding map. This embedding will enable us to prove:
\begin{lemma}\label{lem diff}
Let $d $ be a positive integer. For any integer $ k $ and for any $\alpha \in [0,1]$,   it holds: 
\[ \ord_H \mathcal{F}^{d,k, \alpha} \ge \frac{d}{k + \alpha} \; . \]
\end{lemma}

The above lemma together with  \cref{resultsA} gives the following consequence of  the theorem of Kolmogorov-Tikhomirov:  

\begin{theo}\label{KT+}
Let $d $ be a positive integer. For any integer $ k $ and for any $\alpha \in [0,1]$, it holds: 

\[  \ord_H \mathcal  F^{d,k,\alpha } = \ord_P \mathcal F^{d,k,\alpha} = \underline \ord_B \mathcal F^{d,k,\alpha} =  \overline \ord_B \mathcal F^{d,k,\alpha}  =\frac{d}{k + \alpha} \; .\]\end{theo}

\begin{proof}\label{proof KT+}

First, by \cref{resultsA}, it holds: 
\[       \ord_ H  \mathcal{F}^{d,k, \alpha} \le   \underline \ord_ B  \mathcal{F}^{d,k, \alpha} \le  \overline \ord_ B  \mathcal{F}^{d,k, \alpha}    \qand  \ord_ H  \mathcal{F}^{d,k, \alpha} \le   \ord_ P  \mathcal{F}^{d,k, \alpha} \le  \overline \ord_ B  \mathcal{F}^{d,k, \alpha}     \; .  \] 
From there, by \cref{KT} and \cref{lem diff}, it holds: 
\[     \frac{d}{k + \alpha}  \le  \ord_ H  \mathcal{F}^{d,k, \alpha} \le   \underline \ord_ B  \mathcal{F}^{d,k, \alpha} \le  \overline \ord_ B  \mathcal{F}^{d,k, \alpha}   = \frac{d}{k + \alpha}  \;,  \]
and 

\[  \frac{d}{k + \alpha} \le      \ord_ H  \mathcal{F}^{d,k, \alpha} \le   \ord_ P  \mathcal{F}^{d,k, \alpha} \le  \overline \ord_ B  \mathcal{F}^{d,k, \alpha}   = \frac{d}{k + \alpha}   \; .  \] 
From there, all of the above inequalities are indeed equalities, which gives the desired result. 
\end{proof} 

\subsubsection{Local and global emergence \label{app emergence}}

The framework of scales moreover allows us to answer a problem set by Berger in \cite{berger2020complexities} on the largeness of ergodic decomposition for wild dynamical systems. 
We now consider a compact metric space $ (X,d)$ and a measurable map $ f : X \to X$. 
We denote $\mathcal{M}$ the set of probability Borel measures on $ X$ and $ \mathcal{M}_f $ the subset of $ \mathcal{M}$ containing $f$-invariant measures. The space $\mathcal{M}$ is endowed with the Wasserstein distance $W_1$ defined by: 
\[ W_1( \nu_1,  \nu_2 ) = \sup_{ \phi \in \mathrm{Lip}^1 ( X) } \int \phi d(\nu_1 - \nu_2) \;,  \]
inducing the weak $*$- topology  for which $ \mathcal{M}$  is compact. 
A way to measure the wildness of a dynamical system is to measure how far from being ergodic an invariant measure $ \mu$  is. Then  by Birkhoff's theorem, given a measure $ \mu \in \mathcal{M}_f$, for $\mu$-almost every $x \in X $  the following measure is well defined: 
\[  e^f(x) := \lim_{ n\rightarrow  + \infty}  \frac{1}{n} \sum_{k = 0}^{n-1} \delta_{f^{k} (x) } \;, \]
and also this limit measure is ergodic.
The notion of \emph{emergence}, introduced by Berger, describes the size of the subset of ergodic measures that can be obtained by limits of empirical measures, given an $f$-invariant probability measure on $X$. 
\begin{definition}[Emergence, \cite{berger2017emergence,berger2021emergence}]
The \emph{emergence} of a measure  $\mu \in \mathcal{M}_f$  at $ \epsilon > 0 $ is defined by: 
\[ \mathcal{E}_\mu(\epsilon) = \min  \{  N \in \mathbb{N} \ : \ \exists \nu_1, \dots,  \nu_N \in \mathcal{M}_f,  \ \int_X W_1 ( e^f(x),  \{ \nu_i \}_{1 \le i \le N } ) d \mu (x) \le \epsilon \} \;  . \]
\end{definition}
Note that the emergence is the quantization number of the ergodic decomposition of $ \mu$. 
The case of high emergence corresponds to dynamics where the considered measure is far from being ergodic. The following result shows us that the order is an adapted scaling in the study of the ergodic decomposition. 
\begin{theorem}[ \cite{bolley2007quantitative,kloeckner2015geometric,berger2021emergence} \label{thm BBP}]
Let $(X, d)$ be a metric compact space of finite dimension, then:
\[ \underline \dim_B X  \le \underline \ord_B ( \mathcal{M}) \le \overline \ord_B ( \mathcal{M})  \le  \overline  \dim_B X \; . \]
\end{theorem} 

Given a measure $ \mu \in \mathcal{M}_f$ we define its \emph{emergence order} by: 
\[ \overline{\ord} \mathcal{E}_\mu  := \limsup_{\epsilon \rightarrow 0} \frac{\log \log \mathcal{E}_\mu ( \epsilon ) }{- \log \epsilon} = \inf \left\{ \alpha > 0 : \mathcal{E}_\mu ( \epsilon ) \cdot  \exp( - \epsilon^{-\alpha}) \xrightarrow[\epsilon \rightarrow 0]{}  0 \right\} \; . \]

We  denote $ \mu_{e^f} := {e^f}_* \mu $ the \emph{ergodic decomposition} of $ \mu$; i.e. the push forward by $e^f$ of $ \mu$. A local analogous quantity to the emergence order is the local order of the ergodic decomposition of $ \mu$, for $ \nu \in \mathcal{M}_f$ it is defined by: 

\[ \overline \ord \mathcal{E}^{\loc}_\mu ( \nu )  := \limsup_{\epsilon\rightarrow 0 } \frac{\log - \log ( \mu_{e^f} ( B(\nu, \epsilon )) }{ - \log \epsilon} \; . \]
Berger asked if the following comparison between the asymptotic behavior of the mass of the balls of the ergodic decomposition of $ \mu$ and the asymptotic behavior of its quantization holds. 

\begin{problem}[Berger, {\cite[Pbm 4.22]{berger2020complexities}} \label{prob Berger}]
Let $(X,d)$ be a compact metric space, $f : X \to X$ a measurable map and $\mu$ an $f$-invariant probability measure on $X$. Does the following hold ? 
\[ \int_{\mathcal{M}_f}  \overline \ord  \mathcal{E}_\mu^{\loc} 
d \mu_{e^f}    \le  \overline \ord \mathcal{E}_\mu  \; .   \]

\end{problem}

We propose here a stronger result that answers to the latter problem as a direct application of \cref{resultsC}: 

\begin{theo}\label{ans Berger}

Let $(X,d)$ be a compact metric space, $f : X \to X$ a measurable map and $\mu$ an $f$-invariant probability measure on $X$. For $\mu_{e^f}$-almost every $\nu \in \mathcal{M}$, it holds: 
\[\overline \ord\mathcal{E}^{\loc}_\mu ( \nu ) \le \overline \ord \mathcal{E}_\mu \;  .  \]
\end{theo}

\begin{proof}
Note that  $\overline \ord\mathcal{E}^{\loc}_\mu  = \overline \ord_\loc \mu_{e^f}    $ and $  \overline \ord \mathcal{E}_\mu  = \overline \ord_Q \mu_{e^f}$. Now by \cref{resultsC}, it holds  $\mu_{e^f}$-almost surely  that  $ \overline \ord_\loc \mu_{e^f}     \le   \overline \ord_Q \mu_{e^f}$  which is the desired result. 
\end{proof}
\section{Metric scales} \label{lieu metric scl}
Metric scales will be bi-Lipschitz invariants generalizing Hausdorff, packing and box dimensions of metric spaces. 
Before defining and comparing metric scales we show a handful of basic properties of scalings and present some relevant examples.
\subsection{Scalings}
We first recall that a family  $ \Scl = (\scl_\alpha)_{\alpha \ge 0} $  of positive non-decreasing functions on  $(0, 1)$ is a \emph{scaling} when for any $ \alpha> \beta > 0 $ and any $\lambda > 1 $ close enough to $1$, it holds: 
\begin{equation}\label{def of scl} \tag{$ * $}
    \scl_\alpha  ( \epsilon)  = o \left(  \scl_\beta ( \epsilon^\lambda ) \right)   \qand   \scl_\alpha  ( \epsilon)  = o  \left(   \scl_\beta ( \epsilon ) ^\lambda  \right)   \, \text{ when }  \epsilon \rightarrow 0  \; . 
\end{equation}

An immediate consequence of the latter definition is the following: 
\begin{fact} \label{ineg cons}
Let $ \Scl $ be a scaling then for any $ \alpha > \beta > 0$ and for any constant  $ C > 0 $, for $ \epsilon > 0 $ small enough,  it holds:
\[ \scl_\alpha ( \epsilon ) \le \scl_\beta ( C \cdot \epsilon ) \; . \]
\end{fact}
A consequence  of the latter fact is the following which will allow us to compare the different versions of scales: 
\begin{lemma}
\label{lem inegalite}
Let $f, g : \mathbb R_+^* \to \mathbb R_+^* $ be two functions such that  $f \le g$ near $0$. For any constant $ C > 0$, it holds: 
\[   \inf \left\{ \alpha > 0 : f(  C \cdot \epsilon)  \cdot \scl_\alpha(\epsilon) \xrightarrow[\epsilon \rightarrow 0]{}  0 \right\} \le \inf \left\{ \alpha > 0 : g( \epsilon) \cdot\scl_\alpha(\epsilon) \xrightarrow[\epsilon \rightarrow 0]{} 0 \right\} \]
and
\[  \sup \left\{ \alpha > 0 : f ( C \cdot  \epsilon) \cdot \scl_\alpha(\epsilon) \xrightarrow[\epsilon \rightarrow 0]{}  + \infty \right\} \le \sup \left\{ \alpha > 0 : g( \epsilon) \cdot  \scl_\alpha(\epsilon) \xrightarrow[\epsilon \rightarrow 0]{} +  \infty\right\} \; . \]
\end{lemma}
\begin{proof}
It suffices to observe that, by \cref{ineg cons},  for any $ \alpha > \beta > 0$, and $ \epsilon >  0  $ small, it holds: 
\[ f( \epsilon) \cdot \scl_\alpha ( \epsilon) \le g( \epsilon) \cdot \scl_\beta  ( C^{-1}  \cdot  \epsilon)  = g (  C \cdot \tilde \epsilon ) \cdot \scl_\beta ( \tilde \epsilon )  \; ,  \]
with $ \tilde \epsilon = C \cdot \epsilon $. 
\end{proof}
The following will provide a sequential characterization of scales: 
\begin{lemma}[Sequential characterization of  scales]
\label{contdiscr}
Let $ \Scl$ be a scaling and $f : \mathbb R_+^* \to \mathbb R_+^* $ be a non increasing function. 
Let $(r_n) _{n \ge 1}$ be a sequence of positive real numbers  decreasing to $0$ and such that $ \log r_{n+1} \sim \log r_n$ as $n \rightarrow  + \infty$, then it holds:
\[ \inf \left\{ \alpha > 0 : f( \epsilon) \cdot\scl_\alpha(\epsilon) \xrightarrow[\epsilon \rightarrow 0]{}  0 \right\} = \inf \left\{ \alpha > 0 : f( r_n) \cdot  \scl_\alpha( r_n) \xrightarrow[n \rightarrow +  \infty]{} 0 \right\}  \; \]
and 
\[ \sup\left\{ \alpha > 0 : f( \epsilon) \cdot \scl_\alpha(\epsilon) \xrightarrow[\epsilon \rightarrow 0]{}  +  \infty \right\} =  \sup\left\{ \alpha > 0 : f( r_n ) \cdot \scl_\alpha(r_n ) \xrightarrow[n \rightarrow  + \infty]{}   + \infty  \right\} \; .\]
\end{lemma}
\begin{proof}
Fix $\alpha > 0$ and $\epsilon > 0$. If $\epsilon$ is sufficiently small, there exists an integer \(n > 0\) such that \(r_{n+1} < \epsilon \le r_n\). Since \(f\) is non-increasing and \(\scl_\alpha\) is increasing, we have the inequalities:
\begin{equation} \label{ineq:first}
f(r_n) \cdot \scl_\alpha(r_{n+1}) \leq f(\epsilon) \cdot \scl_\alpha(\epsilon) \leq f(r_{n+1}) \cdot \scl_\alpha(r_n).
\end{equation}
Now, let \(\beta\) and \(\gamma\) be positive real numbers such that \(0 < \beta < \alpha < \gamma\). For \(\lambda\) sufficiently close to \(1\) and for sufficiently small \(\epsilon\), by  \cref{def of scl} it holds:
\begin{equation} \label{ineq:scaling}
\scl_\gamma(r_n) \leq \scl_\alpha(r_n^\lambda) \quad \text{and} \quad \scl_\alpha(r_n) \leq \scl_\beta(r_n^\lambda) \; .
\end{equation}
As $ \lambda \ge \frac{\log r_{n+1}}{\log r_n} $ for large $n$, it holds: 
\[
r_n^\lambda \leq r_{n+1} \; .
\]
This together with \cref{ineq:scaling} implies: 
\begin{equation} \label{ineq:scaledineqs}
\scl_\gamma(r_n) \leq \scl_\alpha(r_{n+1}) \quad \text{and} \quad \scl_\alpha(r_n) \leq \scl_\beta(r_{n+1}).
\end{equation}
By combining \cref{ineq:first,ineq:scaledineqs}, we obtain:
\[
f(r_n) \cdot \scl_\gamma(r_n) \leq f(\epsilon) \cdot \scl_\alpha(\epsilon) \quad \text{and} \quad f(\epsilon) \cdot \scl_\alpha(\epsilon) \leq f(r_{n+1}) \cdot \scl_\beta(r_{n+1}).
\]
Thus, it follows that:
\[
\limsup_{n \to +\infty} f(r_n) \cdot \scl_\gamma(r_n) \leq \limsup_{\epsilon \to 0} f(\epsilon) \cdot \scl_\alpha(\epsilon) \leq \limsup_{n \to +\infty} f(r_n) \cdot \scl_\beta(r_n),
\]
and similarly:
\[
\liminf_{n \to +\infty} f(r_n) \cdot \scl_\gamma(r_n) \leq \liminf_{\epsilon \to 0} f(\epsilon) \cdot \scl_\alpha(\epsilon) \leq \liminf_{n \to +\infty} f(r_n) \cdot \scl_\beta(r_n) \; .
\]
Since this holds for every positive \(\alpha\), and since \(\beta\) and \(\gamma\) can be taken arbitrarily close to \(\alpha\), we obtain the desired result.
\end{proof}
The following provides a whole class of scalings. It shows in particular that the families brought in \cref{example of growth} are indeed scalings.
\begin{proposition} \label{ex scl}
For any integers $ p ,  q \ge 1$, the family $  \Scl^{p,q}  = ( \scl^{p,q}_\alpha ) _{\alpha  > 0 }$ defined for any $ \alpha > 0$ by: 
\[ \scl^{p,q}_\alpha : \epsilon \in ( 0,1 ) \mapsto \frac{1}{\exp ^{\circ p  } ( \alpha  \cdot  \log_+^{ \circ q } ( \epsilon^{-1 } ) ) } \]
is a scaling; where $ \log_+ : t \in \R \mapsto \log(t) \cdot \1_{ t > 1  } $ is the positive part of the logarithm. 
\end{proposition}
We prove this proposition below. Note in particular that:  \[ {\Scl^{1,1} =  \dim = ( \epsilon \in ( 0,1) \mapsto \epsilon^{ \alpha } ) _{\alpha > 0 }} \qand \Scl^{2,1} =  \ord = ( \epsilon \in ( 0, 1 )  \mapsto \exp ( - \epsilon^{ -\alpha} ) ) _{ \alpha > 0 }\;\] are both scalings. 
Let us give an example of a space that has finite box scales for the scaling $ \Scl^{2,2} $ as defined in \cref{ex scl}.  Consider the space  $A$ of holomorphic functions on the disk $ \mathbb D (R) \subset \mathbb{C} $ of radius $ R > 1$ which are bounded by $1$:
\[ A = \left\{ \phi = \sum_{n \ge 0} a_n z^n  \in C^\omega ( \mathbb D ( R), \C ) : \sup_{\mathbb D (R) } \vert \phi \vert \le 1 \right\}  \text{endowed with the norm} \| \phi \|_\infty  := \sup_{ z \in \D (1) } \vert \phi (z) \vert \; . \]
Kolmogorov and Tikhomirov gave the following estimate of its covering number: 
\begin{theorem}[Kolmogorov, Tikhomirov {\cite{tikhomirov1993varepsilon}[Equality (129)]} \label{holo}]
The following estimate on the covering number of  ${(A, \| \cdot \| _ \infty )}  $ holds: 
\[ \log   \mathcal{N}_\epsilon (A) = ( \log R )^{-1}\cdot  \vert  \log \epsilon  \vert ^2  + O ( \log \epsilon^{-1}  \cdot \log \log \epsilon^{-1}  ) , \ \text{when $ \epsilon $ tends to $0$. }  \;.    \]
\end{theorem}
In the framework of scales, the above translates as: 
\[ \underline \Scl^{2,2}_B A = \overline \Scl^{2,2}_B A = 2 \; .  \]
Let us now give a proof of  \cref{ex scl}.
\begin{proof}[Proof of \cref{ex scl}]
The proof is based on the following two facts:
 \begin{fact} \label{fact log}
 For every $\nu > 1 $, for every $ d \ge 1 $ and  for $ y > 0 $ large enough, it holds: 
 \[ \log^{\circ d } ( y^ \nu ) \le \nu \cdot  \log^{\circ d } ( y)  \; . \]
 \end{fact}
\begin{fact}  \label{f1 ex}
For any $ \gamma > 0$ and $ \nu > 1 $ close to $1$, it holds for $ \epsilon > 0 $ small:   
\[  \scl^{p,q} _{ \nu \cdot \gamma  } (\epsilon )  \le  \scl^{p,q}_\gamma  ( \epsilon^\nu)    \qand  \scl^{p,q} _{ \nu \cdot \gamma  } (\epsilon )  \le   \scl^{p,q}_ {\gamma }( \epsilon )    ^\nu  \; .   \]
\end{fact}
Actually \cref{f1 ex} will be proved using \cref{fact log}. First let us show recursively: 

\begin{proof}[Proof of \cref{fact log}] 
We prove this fact by induction on $d \ge 1$. For $d = 1$ note that the inequality is obvious as the equality holds. Then we conclude by induction on $d$  based on the following:
 \[ \log^{\circ (d+1) } ( y^ \nu )  = \log ( \log^{\circ d } ( y^\nu) )  \le \log ( \nu \log^{\circ d} y )  = \log \nu + \log^{\circ(d+1)} y \;  , \]
where the inequality is given by  the induction hypothesis.
\end{proof}
We are now ready to show: 
\begin{proof}[Proof of \cref{f1 ex}]
We apply \cref{fact log} with $d = q $ and $ y = \epsilon^{-1} $ for sufficiently small values of $ \epsilon$ to obtain for every $ \nu > 1$: 
\begin{equation*} \label{app fact log} 
\log^{ \circ q} ( \epsilon^{-\nu} )  \le \nu \cdot \log^{\circ q} ( \epsilon^{-1} )   \; .
\end{equation*}  
Multiplying by $ \gamma $ and composing by $ t \mapsto 
1 \slash \exp^{\circ p } (t)  $ which is decreasing, yields the first inequality in \cref{f1 ex}. 
 
To show the second inequality, we apply again \cref{fact log} with $d = p$ and  $ y =  \frac{1}{\scl_\gamma^{p,q} ( \epsilon)} $ which is large for small values of $ \epsilon >0$ to obtain: 
\[ \log^{\circ p } ( y^\nu )  \le \nu \cdot   \log^{\circ p } ( y  ) \; . \]
As $ y = \exp^{\circ p} ( \gamma  \cdot \log^{\circ q} ( \epsilon^{-1} ) )$, the above inequality translates as $ \log^{\circ p } ( y^\nu) \le \nu \cdot \gamma \cdot \log^{ \circ q} ( \epsilon^{-1} ) $. It follows that:
\[  y^\nu \le \exp^{\circ p } ( \nu \cdot \gamma \cdot \log^{ \circ q} ( \epsilon^{-1} ) ) =   \frac{1}{ \scl_{ \nu \cdot \gamma} ^{p,q}     ( \epsilon) } \; . \] 
In other words, it holds $\scl_{\nu \cdot \gamma}^{p,q} ( \epsilon)  \le  y^{-\nu} $ which is exactly the second inequality of \cref{f1 ex}. 
\end{proof}
We are now ready to prove the two estimates of \cref{def scl} in the definition of scaling for the family $ \Scl^{p,q}$. 
Let us fix $ \alpha > \beta > 0 $ and $ \lambda >1 $ such that $ \alpha > \lambda^2 \cdot \beta$. As $ \scl_\alpha^{p,q}$ is  decreasing with $ \alpha$, it holds: 
\begin{equation} \label{pq} 
\scl_{ \alpha} ^{p,q} ( \epsilon )   \le \scl_{\lambda^2  \cdot \beta } ^{p,q}  ( \epsilon ) \; . 
\end{equation}
On the other hand, by \cref{f1 ex} we have: 
\begin{equation} \label{pq2} 
\scl^{p,q} _{\lambda^2  \cdot \beta }  ( \epsilon )  \le   \scl^{p,q} _{\lambda  \cdot \beta }  ( \epsilon ^\lambda )  \le  \left(  \scl^{p,q} _{ \beta }  ( \epsilon ^\lambda) \right)   ^\lambda \qand  \scl^{p,q} _{\lambda^2  \cdot \beta }  ( \epsilon )  \le \left(  \scl^{p,q} _{ \beta }  ( \epsilon ) \right)  ^{\lambda^2}  \;.
\end{equation}
The above \cref{pq,pq2} imply: 
\[
\frac{\scl^{p,q}_{\alpha}(\epsilon)}{\scl^{p,q}_{\beta}(\epsilon^\lambda)} \leq \left( \scl^{p,q}_{\beta}(\epsilon^\lambda) \right)^{\lambda - 1}  \quad \text{and} \quad
\frac{\scl^{p,q}_{\alpha}(\epsilon)}{\scl^{p,q}_{\beta}(\epsilon)^\lambda} \leq \left( \scl^{p,q}_{\beta}(\epsilon) \right)^{\lambda \cdot (\lambda - 1)}  
 \; ,  \]
which both converge to $ 0 $  as $ \epsilon$ goes to $ 0 $, providing the desired result. 
\end{proof} 
 \subsection{Box scales}  
As introduced in \cref{def box}, \emph{lower and upper  box scales} of a metric space $ (X,d) $ are defined by: 
\[\underline \Scl_B X=  \sup \left\{ \alpha > 0 :  \mathcal{N}_\epsilon (X) \cdot \scl_\alpha ( \epsilon )\xrightarrow[\epsilon \rightarrow 0]{}   + \infty  \right\} \] and 
\[\overline {\Scl}_B   X =   \inf \left\{ \alpha > 0 : \mathcal{N}_\epsilon (X)  \cdot  \scl_\alpha ( \epsilon ) \xrightarrow[\epsilon \rightarrow 0]{}   0  \right\}  \;, \] 
where the covering number $ \mathcal{N}_{\epsilon} (X)$ is the minimal cardinality of a covering of $X$ by balls with radius $ \epsilon > 0 $. 

In general, the upper and lower box scales must not coincide. We will give such examples for the order in \cref{ex inho}. Now we list a few properties of box scales that are well known in the specific case of dimension. 
\begin{fact} \label{fact basic pties}
Let $ (X,d)$ be a metric space. The following properties hold true: 
\begin{enumerate}
\item  if $\underline{\Scl}_B ( X) <  + \infty $, then $(X,d)$ is totally bounded, 
\item for every subset $ E \subset X $ it holds $ \underline{\Scl}_B E \le \underline {\Scl}_B X $ and $ \overline{\Scl}_B E \le \overline {\Scl}_B X $, 
\item for every subset $ E \subset X$ it holds $ \underline{\Scl}_B E  =  \underline {\Scl}_B \cl( E )   $ and $ \overline{\Scl}_B E  =  \overline {\Scl}_B \cl ( E )   $. 
\end{enumerate}
\end{fact}

Observe that $1.$ and $2.$ are direct consequences of the definitions. 

To see $ 3.$ it is enough to observe that $  \mathcal{N}_{  \epsilon}  ( E) \le   \mathcal{N}_\epsilon ( \cl ( E )  )  \le   \mathcal{N}_{ \epsilon / 2 } ( E) $ for every $ \epsilon >  0 $. 

As for box dimensions, box scales can also be defined by replacing the covering number by packing number: 
 \begin{definition}[Packing number]\label{def pack number}
For $ \epsilon >0$ let  $\tilde {\mathcal{N}}  _\epsilon (X) $ be the \emph{packing} number of the metric space $ (X,d)$. It is the maximum cardinality of  an $ \epsilon$-separated set of points in $X$ for the distance $d$: 
\[ \tilde {\mathcal{N}}_{\epsilon} ( X) = \sup \{ N \ge 0 : \exists x_1 , \dots x_N \in X ,   d (x_i , x_j) \ge \epsilon  \text{ for every  $ 1 \le i < j \le N $ } \} \; .  \]
\end{definition}
A well known comparison between packing and covering numbers is the following: 
\begin{lemma} \label{cov pack}
Let $(X,d) $ be a metric space. For every $ \epsilon > 0$, it holds: 
\[ \tilde {\mathcal{N}}  _{2\epsilon}  (X)  \le \mathcal{N}  _\epsilon (X)  \le \tilde {\mathcal{N}}  _\epsilon (X)  \; .  \]
\end{lemma}
In virtue of the basic properties of scalings, the covering number can be replaced by the packing number in the definitions of box scales: 
\begin{lemma}\label{packing numb}
Let $(X,d) $ be a metric space and $ \Scl $ a scaling, then box scales of $X$ can be written as: 
\[\underline {\Scl}_B   X =   \sup \left\{ \alpha > 0 :  \tilde {\mathcal{N}}_\epsilon (X)  \cdot  \scl_\alpha ( \epsilon ) \xrightarrow[\epsilon \rightarrow 0]{}    + \infty \right\}   \] and \[  \overline {\Scl}_B   X =   \inf \left\{ \alpha > 0 : \tilde {\mathcal{N}}_\epsilon (X)  \cdot  \scl_\alpha ( \epsilon ) \xrightarrow[\epsilon \rightarrow 0]{}  0  \right\}  \;  .\] 
\end{lemma}
The proof is provided by direct application of \cref{cov pack,lem inegalite}. 
\begin{remark}
\label{other expr}
Another property for the scaling $ \Scl^{p,q}$ from  \cref{ex scl}, with $ p ,  q \ge 1 $, is that the upper and lower box scales  for a metric space $ ( X,d)$ can be written as: 
 \[ \underline \Scl^{p,q} _B ( X ) =  \liminf_{\epsilon \rightarrow 0 } \frac{\log^{\circ p}  ( \mathcal{N}_\epsilon ( X)  )}{ \log^{\circ q } ( \epsilon^{-1} ) } \]
 and 
  \[ \overline \Scl^{p,q} _B ( X ) =  \limsup _{\epsilon \rightarrow 0 }  \frac{\log^{\circ p}  ( \mathcal{N}_\epsilon ( X)  ) }{ \log^{\circ q } ( \epsilon^{-1} ) }  \; . \]
 In particular, for dimension and order:  
   \[ \underline \dim _B ( X ) =  \liminf_{\epsilon \rightarrow 0 } \frac{\log ( \mathcal{N}_\epsilon ( X)  )}{   - \log \epsilon }   \quad, \quad  \overline \dim _B ( X ) =  \limsup _{\epsilon \rightarrow 0 }  \frac{\log ( \mathcal{N}_\epsilon ( X)  ) }{  - \log \epsilon }  \; , \]
 and 
 \[ \underline \ord _B ( X ) =  \liminf_{\epsilon \rightarrow 0 } \frac{ \log \log ( \mathcal{N}_\epsilon ( X)  )}{   - \log \epsilon}   \quad, \quad  \overline \ord _B ( X ) =  \limsup _{\epsilon \rightarrow 0 }  \frac{ \log \log ( \mathcal{N}_\epsilon ( X)  ) }{   - \log \epsilon }  \; . \] 
The above equalities coincide with the most usual definitions of  box dimensions and orders.
\end{remark}
\subsection{Hausdorff scales}
The definition of Hausdorff scales, generalizing Hausdorff dimension, is introduced here using the definition of Hausdorff outer measure as given for instance by Tricot  in \cite{tricot1982two}. We still consider a metric space $(X,d)$. Given a non-decreasing function $\phi \in C(\R_+^*,  \R_+^*)$, such that $ \phi ( \epsilon ) \rightarrow 0$ when $ \epsilon \rightarrow 0$, we recall: 
\[ \mathcal{H}_\epsilon^\phi(X)   := \inf_{ J \text{ countable set} }\left\{ \sum_{j \in J } \phi( \vert B_j \vert)   :  X = \bigcup_{j \in J} B_j, \ \forall j \in J :   \vert  B_j \vert  \le \epsilon \right\} \; ,  \] 
where  $ \vert B \vert$ is the radius of a ball $B \subset X$.  
A countable family $( B_j ) _{j \in J }$ of balls with radius at most $ \epsilon > 0$  such that $ X = \bigcup_{j \in J } B_j$ will be called an \emph{$\epsilon$-cover}  of $X$. \footnote{Note that the historical construction of the Hausdorff measures uses subsets of $X$ with diameter at most $ \epsilon $  instead of the balls with radius at most $ \epsilon $. However, both these constructions  lead to the same definitions of Hausdorff scales.}
Since the set of $\epsilon$-cover is non-decreasing for inclusion as $ \epsilon$ decreases to $ 0$, the following limit does exist:
\[ \mathcal{H}  ^\phi( X) := \lim_{ \epsilon \rightarrow 0} \mathcal{H}_\epsilon^\phi ( X)  \; . \]
Now replacing $(X,d)$ in the previous definitions by any subset $ E$ of $X$ endowed with the same metric $d$,  we observe that $ \mathcal{H}^\phi$ defines an outer-measure on $X$. It is usually called the $ \phi$-Hausdorff measure on $X$. We now introduce the following: 
\begin{definition}[ Hausdorff scale] \label{H scale}
 The \emph{Hausdorff scale} of a metric space  $(X,d)$ is defined by: 
\[\Scl_H  X  = \sup \left\{ \alpha > 0  : \mathcal{H}^{\scl_\alpha}(X) = + \infty  \right\} = \inf \left\{ \alpha > 0  : \mathcal{H}^{\scl_\alpha}(X) = 0   \right\} \; .\]
\end{definition}
Note that the above definition gives us two quantities that are a priori not equal. However, the mild assumptions in the definition of scaling allow us to verify that they indeed coincide and allow us to use the machinery of Hausdorff outer measure to define metric invariants generalizing Hausdorff dimension.  Thus  scalings allow us to have some consistent extension of the definition of Hausdorff dimension. 
\begin{proof}[Proof of the equality in \cref{H scale}] 
It is clear from its definition that $ \alpha  \mapsto \mathcal{H}^{\scl_\alpha}(X)$ is non-increasing. It is then enough to check that if there exists  $\alpha > 0$ such that $ 0 <  \mathcal{H}^{\scl_\alpha}(X) < +  \infty$ then,  for any positive $\delta < \alpha  $, it holds: 
\[ \mathcal{H}^{\scl_{\alpha+ \delta }}(X) = 0 \qand \mathcal{H}^{\scl_{\alpha- \delta }}(X) = + \infty  \; .  \]
Let us fix $\eta >0$, by  \cref{scaling} ,  for $ \epsilon > 0$ small it holds: 
\[ \scl_{\alpha + \delta }(\epsilon) \le \eta \cdot \scl_{\alpha}(\epsilon) \qand   \scl_{\alpha }(\epsilon) \le \eta \cdot  \scl_{\alpha- \delta }(\epsilon)  \; .  \] 
Since $ \epsilon $ is small, it holds: 
\[ 0 < \frac{1}{2} \mathcal{H}^{\scl_\alpha}(X) \le  \mathcal{H}_\epsilon^{\scl_{\alpha}} (X) \le  \mathcal{H}^{\scl_{\alpha}}(X) < +  \infty \; . \]
Given $ ( B_j )_{j\in J } $ an  $\epsilon$-cover of $X$, the following holds: 
\[ \frac{1}{2}\mathcal{H}^{\scl_{\alpha}}(X) \le \mathcal{H}^{\scl_{\alpha}}_\epsilon (X) \le  \sum_{j \in J} \scl_\alpha ( \vert B_j \vert )  \;,  \]  
and then:
\[  \frac{1}{2\eta }  \mathcal{H}^{\scl_{\alpha}}(X)  \le  \frac{1}{\eta }\sum_{j \in J} \scl_{\alpha } ( \vert B_j \vert ) \le  \sum_{j \in J } \scl_{\alpha - \delta } ( \vert B_j \vert )  \; . \]
Since this holds for every $ \epsilon$-cover, the latter inequality leads to: 
\[ \frac{1}{2\eta }  \mathcal{H}^{ \scl_\alpha} (X) \le \mathcal{H}_\epsilon^{\scl_{\alpha- \delta}}(X) \;  , \]
and so:
\begin{equation} \label{h1} 
\frac{1}{2\eta }  \mathcal{H}^{ \scl_\alpha} (X) \le \mathcal{H}^{\scl_{\alpha- \delta}}(X) \;  .
\end{equation} 
On the other side,  there exists an $\epsilon$-cover $ (B_j)_{j \in J }$ of $E$ such that: 
\[ \sum_{j \in J } \scl_{\alpha } ( \vert B_j \vert ) \le 2 \mathcal{H}_\epsilon ^{\scl_{\alpha}}(X) \; . \] 
Now since $ \mathcal{H}_\epsilon ^{\scl_{\alpha}}(X) \le  \mathcal{H}^{\scl_{\alpha}}(X) $, this leads to: 
\[\sum_{j \in J} \scl_{\alpha + \delta} ( \vert B_j \vert ) \le \eta \cdot \sum_{j} \scl_{\alpha } ( \vert B_j \vert ) \le  2 \eta  \cdot \mathcal{H}^{\scl_{\alpha}}(X) \; .\]
From there: 
\begin{equation}\label{h2} 	
\mathcal{H}_\epsilon ^{\scl_{\alpha+ \delta }}(X) \le 2 \eta  \cdot  \mathcal{H}^{\scl_{\alpha}}(X) \; ,
\end{equation}
and this holds for every small $ \epsilon$. 
Taking the limit as $ \epsilon \to  0 $ in \cref{h1,h2} gives: 
\[ \frac{1}{2\eta }  \mathcal{H}^{ \scl_\alpha} (X) \le \mathcal{H}^{\scl_{\alpha- \delta}}(X) \qand \mathcal{H}_\epsilon ^{\scl_{\alpha+ \delta }}(X) \le 2 \eta  \cdot  \mathcal{H}^{\scl_{\alpha}}(X)  \; . \]
To conclude, note that as the latter holds for  $\eta$ arbitrarily small, it follows that   $ \mathcal{H}^{\scl_{\alpha- \delta }} (X) =   + \infty$ and $ {\mathcal{H}^{\scl_{\alpha+ \delta}} (X) =  0}$.
\end{proof}
As box scales, Hausdorff scales are non-decreasing for inclusion. We will see a stronger property of Hausdorff scales in \cref{countstab}. 
\subsection{Packing scales}
\subsubsection{Packing scales through modified box scales}
The original construction of packing dimension relies on the packing measure introduced by Tricot in \cite{tricot1982two}. We first define packing scales by modifying upper box scales and we show later how they are related to packing measures.
\begin{definition}[Packing scale] \label{packing l box}
Let $(X,d)$ be a metric space and $ \Scl$ a scaling. The \emph{packing scale} of $X$ is defined by: 
\[ \Scl_P X = \inf \left\{   \sup_{ n \ge 1} \overline{\Scl}_B E_n  :  ( E_n)_{n \ge 1} \in X^{\N^*}  \text{ s.t. }\bigcup_{n \ge 1} E_n = X  \right\} \; . \]
\end{definition}
The following comes directly from the definition of packing scale:
\begin{proposition} \label{Falconer paquet boite}
Let $(X,d)$ be a metric space and let $ \Scl$ be a scaling. It holds:
\[ \Scl_P X \le \overline{\Scl}_B X  \; . \]
\end{proposition}

\subsubsection{Packing measures}
In this paragraph we show the relationship between packing measures and packing scales. Let us first recall a few definitions. 

Given $ \epsilon >0$, an \emph{$\epsilon$-pack} of a metric space $ (X,d)$  is a countable collection of disjoint balls  of $X$ with radii at most $ \epsilon$.  As for Hausdorff outer measure, consider  $\phi: \mathbb R _+^* \to \mathbb R _+^* $ an non-decreasing function such that $ \phi( \epsilon) \rightarrow 0$ as $ \epsilon \rightarrow 0$. For $ \epsilon> 0$, put:   
\[ \mathcal{P}_\epsilon ^{\phi}(X) := \sup \left\{ \sum_{i \in I } \phi (  \vert B_i \vert  ) : (B_i )_ { i \in I  }\ \text{ is an $ \epsilon$-pack of $X$}\right\} . \]
Since $ \mathcal{P}_\epsilon^\phi(X)$ is non-increasing when $\epsilon$ decreases to $0$,  the following quantity is well defined: 
\[ \mathcal{P}_0^\phi (X) := \lim_{ \epsilon  \rightarrow 0} \mathcal{P}_\epsilon^\phi(X).\]
The idea of Tricot is to build an outer measure from this quantity: 
\begin{definition}[Packing  measure]For every subset $ E $ of $X$ endowed with the same metric $d$, the \emph{packing  $\phi$-measure} of $E$ is defined by: 
\[ \mathcal{P}^\phi (E) = \inf  \left\{ \sum_{n \ge  1 } \mathcal{P}_0^{\phi} ( E_n ) : E =  \bigcup_{n \ge 1 } E_n  \right\} \; . \]\end{definition}
Note that 
$\mathcal{P}^\phi$ is an outer-measure on $X$ and can eventually be infinite or null. The following shows the equivalence of Tricot's counterpart definition of the packing scale; this will be useful to show the equality between the upper local scale and the packing scale of a measure given by  \cref{resultsC} eq. (c\&g).
 
\begin{proposition} \label{car packing}
The packing scale of a metric space $ (X,d)$ verifies: 
\[ \sup \left\{ \alpha > 0  : \mathcal{P}^{\scl_\alpha} (X) = +  \infty \right\} = \Scl_P X =\inf  \left\{ \alpha > 0 : \mathcal{P}^{\scl_\alpha} (X) = 0 \right\}   \; . \]
\end{proposition}
\begin{proof} Let $ ( E_n)_{n \ge1 } $ be a family of subsets of $ X $. Since each map $ \alpha  \mapsto \mathcal{P}_0^{\scl_\alpha} ( E_n ) $ is non-increasing and non negative, we have: 
\begin{equation}\label{eq up box 0}  \inf \left\{ \alpha > 0 : \sum_{n \ge 1} \mathcal{P}_0^{\scl_\alpha}  ( E_n) = 0     \right\}  = \sup_{n \ge 1}  \inf \left\{ \alpha > 0 :  \mathcal{P}_0^{\scl_\alpha}  ( E_n) = 0    \right\} \;   .  \end{equation} 
We decompose the proof into two intermediary steps given by the following lemmas:
\begin{lemma}\label{car up box lem2}  Given $ \alpha > 0, $ if $ \mathcal{P}_0^{\scl_\alpha} (E)$ is finite, then for every  $ \delta \in(0, \alpha) $, it holds:
\[ \mathcal{P}_0^{\scl_{ \alpha + \delta}} (E) = 0 \qand  \mathcal{P}_0^{\scl_{ \alpha - \delta}}  (E) = + \infty \; .\]
\end{lemma}
\begin{lemma}\label{car up box lem3}
For every $E\subset X$, it holds:
\[ \sup \left\{ \alpha > 0 : \mathcal{P}_0^{\scl_\alpha} ( E ) = + \infty \right\} = \overline \Scl_B E = \inf  \left\{ \alpha > 0 :  \mathcal{P}_0^{\scl_\alpha} ( E ) = 0 \right\} \; . \] 
\end{lemma}
\cref{car up box lem2} will allow proving \cref{car up box lem3}. 
Before proving these lemmas let us see how they allow us to conclude the proof of \cref{car packing}.
First note that the second equality of \cref{car up box lem2} implies:
\begin{equation*}\label{eq up box 1}  \sup \left\{ \alpha > 0 : \sum_{n \ge 1} \mathcal{P}_0^{\scl_\alpha}  ( E_n) = + \infty     \right\} 
=  \sup_{n \ge 1}  \sup \left\{ \alpha > 0 :  \mathcal{P}_0^{\scl_\alpha}  ( E_n) =+\infty    \right\} \; .     
\end{equation*}
Consequently by \cref{car up box lem3} it holds: 
\[  \sup \left\{ \alpha > 0 : \sum_{n \ge 1} \mathcal{P}_0^{\scl_\alpha}  ( E_n) = + \infty     \right\} =  \sup_{n \ge 1} \overline{\Scl}_B E_n = \inf \left\{ \alpha > 0 : \sum_{n \ge 1} \mathcal{P}_0^{\scl_\alpha}  ( E_n) = 0    \right\} \; .   \] 
Taking the infimum over families $(E_n)_{n \ge 1} $ that cover $ X $, we obtain the desired result. 
\end{proof}
It remains to show \cref{car up box lem2,car up box lem3}.
\begin{proof}[Proof of \cref{car up box lem2}] 
Given $ \eta > 0$, by \cref{scaling} of scaling, for $ \epsilon > 0 $ small enough, it holds: 
\begin{equation} \label{eq pcar}
\scl_{\alpha + \delta}(\epsilon) \le \eta \cdot  \scl_{\alpha}(\epsilon) \qand   \scl_{\alpha }(\epsilon) \le  \eta \cdot \scl_{\alpha- \delta}(\epsilon)  \; .
\end{equation}  
Let  $(  B_j) _{j \ge  1 } $ be  an  $\epsilon$-pack of $E$. Then, by the above \cref{eq pcar} it holds: 
\[ \eta^{-1} \cdot \sum_{ j \ge 1} \scl_{\alpha +\delta} ( \vert B_j \vert ) \le  \sum_{ j \ge 1} \scl_{\alpha } ( \vert B_j \vert )  \le  \eta  \cdot \sum_{ j \ge 1} \scl_{\alpha  - \delta} ( \vert B_j \vert )  \; .  \] 
As this holds for every $ \delta $ - pack of $E$, it follows: 
\[ \eta^{-1} \cdot  \mathcal{P}_\epsilon^{\scl_{\alpha + \delta} } (E) \le    \mathcal{P}_\epsilon^{\scl_{\alpha } } (E)  \le \eta \cdot  \mathcal{P}_\epsilon^{\scl_{\alpha - \delta} } (E)   \; . \]
Taking the limit as $ \epsilon$ goes to $ 0$ and $ \eta$  arbitrary small allows us to conclude. \end{proof}

\begin{proof}[Proof of \cref{car up box lem3}] By \cref{car up box lem2}, it suffices to show that:  
\begin{equation}\label{cara final}\sup \left\{ \alpha > 0 : \mathcal{P}_0^{\scl_\alpha} ( E ) = + \infty \right\} \le  \overline \Scl_B E  \le  \inf  \left\{ \alpha > 0 :  \mathcal{P}_0^{\scl_\alpha} ( E ) = 0 \right\} \; . \end{equation} 
We start with the second inequality. If $ \mathcal{P}_0^{\scl_\alpha} ( E)  >  0$ for every $ \alpha > 0 $, there is nothing to prove. Thus consider $ \alpha > 0$ such that $\mathcal{P}_0^{\scl_\alpha} ( E ) =  0 $. Then for every $ \epsilon > 0 $ sufficiently small it holds $ \mathcal{P}_\epsilon^{\scl_\alpha} ( E ) < 1  $. In particular, the packing number (see def. \ref{def pack number}) satisfies $ \tilde{ \mathcal{N}}_\epsilon ( E) \cdot \scl_\alpha ( \epsilon) < 1 $. By \cref{packing numb} we obtain $ \overline{\Scl}_B E \le \alpha$  which gives the second inequality by taking the infima of such $ \alpha > 0$. 

To prove the first inequality, assume that there exists $ \alpha > 0 $ such that  $ \mathcal{P}_0^{\scl_\alpha} ( E ) = + \infty$, otherwise there is nothing to prove. For such an $ \alpha $ and for every  $ \epsilon > 0 $ there exists an $ \epsilon$-pack $ ( B_j ) _{ j \ge 1} $ such that: 
\begin{equation} \label{eq p1}
\sum_{j \ge 1} \scl_\alpha ( \vert B_j \vert ) > 1  \; . 
\end{equation} 
For an integer $ k \ge 1 $, denote: 
\begin{equation} \label{def nk}
n_k := \Card \left\{ j \ge 1 : 2^{-(k+1)} \le \scl_\alpha ( \vert B_j \vert ) < 2^{-k} \right\} \; . 
\end{equation} 
Since $\scl_\alpha$ is non-decreasing, it holds:  
\[ \sum_{k \ge 1}  n_k  \cdot 2^{-k }   \ge \sum_{ j \ge 1}  \scl_\alpha ( \vert B_j \vert )   \; . \]
Thus by \cref{eq p1}, it holds: 
\begin{equation} \label{sum nk}
 \sum_{ k \ge 1} n_k \cdot 2^{-k} > 1 \; .  
\end{equation}
Note that, as $ \vert B_j \vert \le \delta $ for every $ j \ge 1 $ , it holds $ n_k = 0 $ for every $ k < - \log_2 \scl_\alpha(\delta) $. 
Then for $ \delta$ small it holds: 
\begin{fact}
There exists an integer $ j \ge 2$  such that:
\[ n_j > j^{-2} \cdot 2^j  \; . \]
\end{fact} 
\begin{proof}
Otherwise we would have: 
\[ \sum_{k \ge 1}  n_k \cdot 2^{-k } \le   \sum_{k \ge  2 } \frac{  1}{  k ^2  }  <  1 \;  , \]
as $ n_0 = n_1 = 0$ for small $ \delta$, and this contradicts  \cref{sum nk}. 
\end{proof}
This latter fact translates as: there exist at least $ n_j$ disjoint balls with radii at least $ \scl_\alpha^{-1} ( 2^{- ( j +1) } ) $. 
This implies: 
\[ \tilde{\mathcal{N}}_ {  \scl_\alpha^{-1} ( 2^{-(j+1)}) } ( E) \ge  n_j > j^{-2} 2 ^j \; , \]
and moreover:
\[  j \ge - \log_2 \scl_\alpha ( \delta) \; . \]
Since these inequalities hold true for $ \delta$ arbitrarily small, there exists an increasing sequence of integers $( j_n )_{n \ge 1 } $ such that:
\begin{equation} \label{packnumbj}
\tilde {\mathcal{N}}_ {\epsilon_n} (E)   > j_n^{-2} 2 ^{j_n} \quad \text{ where }  \epsilon_n := \scl_\alpha^{-1} ( 2^{-( j_n + 1 )} )  \;   .
\end{equation}
Given a positive $ \beta < \alpha$, by \cref{scaling} of scaling, for $ \lambda > 1 $ close to $1$, it holds: 
\begin{equation} \label{eq scl reform}
\scl_\beta ( \epsilon )  \cdot  \left(  \scl_{\alpha} ( \epsilon)  \right)  ^{ -  \lambda^{-1}  }   \xrightarrow[\epsilon \to 0 ]{} + \infty \; . 
\end{equation} 
On the other hand, given such a  $ \lambda> 1$, for $n$ large enough, it holds: 
\[ j_n^{-2} 2^{j_n} \ge 2^{ \lambda^{-1 } \cdot ( j_n + 1 )  } \; .  \]
It follows by \cref{packnumbj}: 
\[ \tilde {\mathcal{N}}_ {\epsilon_n} (E) \ge   \left(  2^{  -  ( j_n + 1 )  } \right) ^{ -  \lambda^{-1}}  =  \left(  \scl_\alpha ( \epsilon_n)  \right) ^{ - \lambda^{-1 }}  \; .  \]
This together with \cref{eq scl reform} gives: 
\[ \scl_\beta ( \epsilon_n) \cdot \tilde {\mathcal{N}}_ {\epsilon_n} (E)     > \scl_\beta ( \epsilon )  \cdot  \left(  \scl_{\alpha} ( \epsilon)  \right)  ^{ -  \lambda^{-1}  }  \xrightarrow[\epsilon \to 0 ]{} + \infty \; . \]
By \cref{packing numb} we deduce $ \overline\Scl_B E \ge  \beta$. Taking $ \beta $ close to $ \alpha$ ends the proof.
\end{proof}
\subsection{Properties and comparison of scales of metric spaces \label{comp of scales}}
We first give a few basic properties of scales that would allow us to compare them. 
Since both packing and Hausdorff scales are defined via measures, they both are $\sigma$-stable as shown in the following:  
\begin{lemma}
\label{countstab}
Let $ (X,d)$ be a metric space. 
Let  $I $ be a countable set  and $(E_i)_{i \in I}$ a covering of $ X$, then for any scaling $ \Scl$: 
\[  \Scl_H X = \sup_{i \in I}  \Scl_H  E_i \qand \Scl_P X = \sup_{i \in I}  \Scl_P  E_i \; . \]
\end{lemma}
\begin{proof}
The equality on packing scales is clear by definition. Let us prove the equality on Hausdorff scales. 
By monotonicity of the Hausdorff measure, it holds $ \Scl_H X \ge \sup_{ i \in I  } \Scl_H E_i$.  For the reverse inequality, consider $ \alpha > \sup_{i \in I} \Scl_H E_i$,  then for any  $i \in I $  it holds $ \mathcal{H}^{\scl_\alpha} (E_i) = 0$. Thus, it holds: 
\[ \mathcal{H}^{\scl_\alpha} ( X ) \le \sum_{ i \in I } \mathcal{H}^{\scl_\alpha} (E_i)  = 0 \; ,  \]
and then $\Scl_H X \le \alpha$. Since this is true for any $\alpha > \sup_{i \in I }  \Scl_H E_i$, the desired result comes. 
\end{proof}
Note that $\sigma$-stability is not a property of box scales. To see that, it suffices to consider a countable dense subset of a metric space  $ (X,d) $ with positive box scales. This is actually a basic known fact for the specific case of dimension that naturally still holds there. 

The following lemma shows in particular that the above scales are  bi-Lipschitz invariants. 

\begin{lemma}\label{lem lip}
Let $ (X,d)$ and $( Y,d) $ be two metric spaces such that there exists a Lipschitz map ${ f : (X,d) \to (Y,d) }$.
Then for any scaling $ \Scl$, the scales of $f(X) $ are at most the ones of $X$:
\[ \Scl_H f( X)  \le \Scl_H X ; \quad \Scl_Pf( X)  \le \Scl_P X ; \quad \underline \Scl_B f(X)  \le \underline \Scl_B X; \quad \overline \Scl_B f(X) \le \overline \Scl_B X \; . \]
\end{lemma}
\begin{proof}
Let us fix $ \epsilon >0$.  Let $ K > 0$  be a Lipschitz constant for $f$. 
We first show the inequalities on box and packing scales.

Consider a minimal covering $ ( B ( x_j, \epsilon ) )_{ 1 \le j \le N} $ of $  \epsilon$-balls centered in $ X$. it holds: 
\[ f (X) =  f \left(  \bigcup_{j =1 }^N  B ( x_j, \epsilon_j )  \right)  \subset  \bigcup_{j =  1 }^N  B ( f(x_j) , K \cdot   \epsilon_j ) \; .  \]
Then $( B( f(x_j)  , K \cdot \epsilon) )_{ 1 \le j \le N} $ is a covering by $ K \cdot \epsilon$-balls of $ f(X) $. Then $ \mathcal{N}_{K \cdot \epsilon} ( f(X) ) \le \mathcal{N}_\epsilon ( X) $ and all the inequalities on the box and packing scales are immediately deduced from \cref{ineq:scaling}. 
Now for Hausdorff scales, consider a countable set $J$ and $( B (x_j, \epsilon_j) )_{j \in J} $ an $ \epsilon$-cover of $X$. Then it holds:
\[ f(X)  \subset \bigcup_{j \in J } B ( f(x_j) , K \cdot   \epsilon_j ) \; . \]
For any $ \alpha > \beta > 0 $ and $ \delta> 0$ small enough, by \cref{ineg cons} , it holds: 
\[ \scl_\alpha ( \delta) \le \scl_\beta ( K^{-1} \cdot  \delta )  \; . \]
Hence for $ \epsilon $ small, it holds: 
\[ \mathcal{H}_{K\cdot \epsilon}^{\scl_ \alpha }( f(X)  )  \le \sum_{j \in J } \scl_\alpha (  K \cdot \epsilon_j   )   \le \sum_{j \in J } \scl_\beta  (   \epsilon_j   ) \;.   \]
As $ \beta > \Scl_H X $, the $ \epsilon $-cover $( B(x_j, \epsilon_j ))_{j \in J} $ can be chosen such that $ \sum_{j \in J } \scl_\beta ( \epsilon_j) $ is arbitrarily small. Consequently, it holds $ \mathcal{H}^{\scl_\alpha}_{K \cdot \epsilon} ( f(X)) = 0 $, and so $ \Scl_H f(X) \le \alpha $. As $ \alpha$ is arbitrarily close to $ \Scl_H X$, it holds:
\[ \Scl_H f(X) \le \Scl_H X \; . \]
\end{proof}

 As a direct application, we obtain the following: 
\begin{corollary} \label{cor dil}
Let $ (X,d)$ and $( Y,d) $ be two metric spaces. Assume that there exists an embedding $ g  : (Y,\delta) \to (X,d) $ such that $ g^{-1} $ is Lipschitz on $g(X)$. 
Then for every scaling $ \Scl$, the scales of $ Y $ are at most the ones of $X$:
\[ \Scl_H Y   \le \Scl_H X ; \quad \Scl_P Y   \le \Scl_P  X ; \quad \underline \Scl_B Y  \le \underline \Scl_B X; \quad \overline \Scl_B Y \le \overline \Scl_B X \; . \]
\end{corollary}
\begin{proof}
By \cref{lem lip} we have $ \Scl_\bullet Y \le \Scl_\bullet g (Y) $ for any $ \Scl_\bullet \in \{ \Scl_H, \Scl_P  \underline{\Scl}_B , \overline{\Scl}_B \}$. As $g ( Y) \subset X $, we have also $ \Scl_\bullet g(Y) \le \Scl_\bullet X $. 
\end{proof}
Remark that \cref{lem lip} and  \cref{cor dil} hold even for scalings that have sub-polynomial behaviors.

The end of this section consists of comparing the different scales introduced and proving \cref{resultsA}. We start by comparing the Hausdorff with lower box scales. The following proposition  generalizes well known facts on dimension. See e.g.  \cite{falconer2004fractal}[(3.17)].  
\begin{proposition}
\label{Falconer hausdorff boite}
 Let $ (X,d)$ be a metric space and $ \Scl$ a scaling, its Hausdorff scale is at most its lower box scale:
\[ \Scl_H   X\le \underline{\Scl}_B   X \; .\]
\end{proposition}
\begin{proof}
We can assume without any loss that $ (X,d)$ is totally bounded. If $ \Scl_H   X=0$  the inequality obviously holds, thus consider a positive number $ \alpha< \Scl_H X$. For $\delta >0$ small enough,   $\mathcal{H}_\delta^{\scl_\alpha }(X) > 1$. Also there exists a $ \delta$-cover  $( B_j ) _{ 1 \le j \le\mathcal{N}_\delta (X) }$. It verifies:
\[  1 < \sum_{1 \le j \le\mathcal{N}_\delta (F) }\scl_\alpha  ( \vert B_j \vert )  = \mathcal{N}_\delta (X) \cdot \scl_\alpha ( \delta) \; .  \]
From there, it holds $ \underline{\Scl}_B X  \ge \alpha $. We conclude by taking  $\alpha $  arbitrarily close to $\Scl_H X$.
\end{proof}

We have compared Hausdorff and packing scales with their corresponding box scales. It remains to compare each other with the following:  
\begin{proposition}
\label{Hausdorff paquet}
Let $ (X,d)$ be a metric space and $ \Scl$ a scaling. It holds: 
\[ \Scl_H X \le \Scl_P X \; . \]
\end{proposition}
\begin{proof}
By \cref{countstab}, Hausdorff scale is $\sigma$-stable: 
\[\Scl_H X  = \inf_{ \bigcup_{n\ge 1 } E_n = X }  \sup_{n \ge 1 } \Scl_H E_n  \; ,  \]
where the infimum is taken over countable coverings of $X$. 
Moreover, by \cref{Falconer hausdorff boite}, we have: 
\[ \Scl_H E \le \underline{\Scl}_B E \le \overline{\Scl}_B E \; ,  \] 
for any subset $ E $ of $X$. 
It follows then:
\[\Scl_H X  \le  \inf_{ \bigcup_{n\ge 1 } E_n = X }  \sup_{n \ge 1 } \overline{\Scl}_B E_n = \Scl_P X  \; .   \]
\end{proof}
For the sake of completeness we will resume:
\begin{proof}[Proof of \cref{resultsA}]
Let $ (X,d) $ be a metric space and $ \Scl$ a scaling.
By \cref{Falconer hausdorff boite},  \cref{Hausdorff paquet} and \cref{Falconer paquet boite},  it holds respectively: 
\[ \Scl_H X \le \underline \Scl_B X , \quad \Scl_H X \le \Scl_P X \qand  \Scl_H X \le \overline \Scl_B X   \; . \] 
Now since $ \underline \Scl_B X \le \overline \Scl_B X $ obviously holds, we deduce the desired result: 
\[  \Scl_H X \le \Scl_P X  \le \overline \Scl_B X  \qand  \Scl_H X \le \underline \Scl_B X  \le \overline \Scl_B X \; .   \]
\end{proof}
\section{Scales of measures} \label{lieu meas scl}

 In this section we recall the different versions of scales of measures we introduced and show the inequalities and equalities comparing them. In particular, we provide proofs of \cref{resultsB} and \cref{resultsC}. They generalize known results from dimension theory to any scaling and moreover bring new comparisons (see \cref{lem quant boite rep au pb}) between quantization and box scales that were not shown yet for even the case of dimension. 

\subsection{Hausdorff, packing and local scales of measures \label{proof thm B}}

Let us recall the definition of local scales.
Let $ \mu$ be a Borel measure on a metric space $(X,d)$ and $ \Scl $ a scaling.  The \emph{lower and upper scales} of $\mu$ are the functions that map a point $x \in X$ to:
\[\underline{\Scl}_{\loc}  \mu (x)  = \sup \left\{  \alpha > 0 : \frac{\mu \left(   B(x,  \epsilon ) \right) }{ \scl_\alpha  ( \epsilon )}   \xrightarrow[\epsilon \rightarrow 0]{}  0\right \} \] and \[ 
\overline{\Scl}_{\loc}  \mu (x) = \inf \left\{ \alpha > 0 :  
\frac{\mu \left(   B(x,  \epsilon ) \right) }{ \scl_\alpha  ( \epsilon )}   \xrightarrow[\epsilon \rightarrow 0]{}  +\infty
\right\} \;.\]

We shall compare local scales with the followings: 

\begin{definition}[Hausdorff scales of a measure] \label{haus meas def}
Let $ \Scl$ be a scaling and $ \mu$ a non-null Borel measure on  a metric space $(X,d)$. We define the
\emph{Hausdorff and $ * $-Hausdorff scales} of the measure $ \mu$ by: 
\[\Scl_H \mu = \inf_{ E \in \mathcal{B} } \left\{ \Scl_H E : \mu(  E ) >  0 \right\} \qand  \Scl^*_H \mu = \inf_{ E \in \mathcal{B} }  \left\{ \Scl_H E : \mu( X \backslash E ) = 0 \right\} \;,  \]
where $\mathcal{B}$ is the set of Borel subsets of $X$. 
\end{definition} 

\begin{definition}[Packing scales of a measure] \label{pack meas def}
Let $ \Scl$ be a scaling and $ \mu$ a non-null Borel measure on  a metric space $(X,d)$. We define the
\emph{ packing and $ * $-packing  scales} of $ \mu$ by: 
\[\Scl_P \mu = \inf_{ E \in \mathcal{B} } \left\{ \Scl_P E : \mu(  E ) >  0 \right\} \qand  \Scl^*_P \mu = \inf_{ E \in \mathcal{B} }  \left\{ \Scl_P E : \mu( X \backslash E ) = 0 \right\} \; .  \]
\end{definition}
\begin{remark}
In order to avoid excluding the null measure $0$, we set as a convention:
\[  \Scl_H 0 =  \Scl^*_H 0 =  \Scl_P 0 =  \Scl^*_P 0 = 0   \; .\]
\end{remark}
The lemma below will allow us to compare local scales with the other scales of measures. 
\begin{lemma}
\label{infess}
Let $\mu$ be a Borel  measure on $X$. 
Then for any Borel subset $F$ of $X$ such that $\mu(F) > 0$, the restriction $ \sigma $ of $\mu$ to $F$ verifies: 
\[ \infess \overline \Scl_{\loc} \mu \le \infess \overline \Scl_{\loc} \sigma \qand \infess \underline \Scl_{\loc} \mu \le \infess \underline \Scl_{\loc} \sigma.  \]
Moreover, if there exists $\alpha >0$  such that  $ F \subset \left\{ x \in X :   \overline \Scl_{\loc} \mu ( x)  > \alpha \right\}$,  it holds then:
\[   \infess \overline \Scl_{\loc}  \sigma \ge \alpha \; ,  \]
and similarly if $ F \subset \left\{ x \in X :   \underline \Scl_{\loc} \mu ( x)  > \alpha \right\}$, it holds:
\[    \infess \underline \Scl_{\loc}  \sigma \ge \alpha \; .\]
\end{lemma}
\begin{proof}
Consider a point $x \in X$, then for any $\epsilon > 0$,  one has   $ \sigma ( B(x,  \epsilon)) \le \mu( B(x,  \epsilon))$, thus by definition of local scales: 
\[\overline \Scl_{\loc}  \mu \le \overline \Scl_{\loc}  \sigma \quad  \text{and} \quad \underline \Scl_{\loc}  \mu \le \underline \Scl_{\loc}  \sigma \; .\]
Now if there exists $ \alpha > 0 $ such that   $ F \subset \left\{ x \in X :   \overline \Scl_{\loc} \mu ( x)  > \alpha \right\}$, as $ \overline{\Scl}_{\loc } \mu (x) \ge \alpha  $ for $ \mu$-almost every $x $ in $F$, it follows from the above inequality that $ \overline{\Scl}_{\loc} \sigma (x) \ge  \alpha $  for $ \mu$-almost every $x $ in $F$, and thus for $ \sigma$-almost every $x \in X $. It follows  $ \infess \overline \Scl_{\loc}  \sigma \ge \alpha$. And the same holds for lower local scales. 
\end{proof}
The following lemma corresponds to part of the results of \cref{resultsB}. We will prove this lemma later in \cref{lieu proof 335}. First, we will use it to prove \cref{resultsC} in \cref{lieu resultC}. This lemma states that the lower and upper local scales of a measure are, respectively, not greater than the Hausdorff and packing scales of the underlying space:
\begin{lemma}
\label{lemma dim loc}
Let $(X, d)$ be a  metric space and $\mu$ a Borel  measure on $X$. Let $\Scl$ be a scaling. Then it holds: 
\[ \supess \underline{\Scl}_{\loc}\mu  \le  \Scl_H X   \qand \supess \overline{\Scl}_{\loc}\mu  \le \Scl_P X   \; .\] 
\end{lemma}
Note that in the above we can replace $X$ by any of its subsets with total mass. This observation directly implies: 
\begin{corollary} \label{rem 35}Let $(X, d)$ be a  metric space and $\mu$ a Borel  measure on $X$. Let $\Scl$ be a scaling. It holds: 
\[ \supess \underline{\Scl}_{\loc}\mu  \le  \Scl_H ^*  \mu    \qand \supess \overline{\Scl}_{\loc}\mu  \le \Scl_P^*  \mu     \; . \]
\end{corollary}
To prove \cref{resultsC} we need to study quantization scales of measures. 
\subsection{Quantization and box scales of measures}
Let us first recall the definition of quantization scales.
Let $(X,d)$ be a metric space and  $\mu$ a Borel measure on $X$. Given $ \epsilon > 0$, the \emph{ $\epsilon$-quantization number} $\mathcal Q_\mu(\epsilon)$ of $ \mu$ is the  minimal cardinality of a set of points that is on average $\epsilon$-close to any point in $X$: 
\[ \mathcal{Q}_\mu(\epsilon) = \inf  \left\{ N \ge 0 : \exists 
\left\{ c_i \right\}_{i= 1, \dots, N } \subset X,  \int_X d(x,  \left\{ c_i\right\}_{  1 \le i \le N } ) d \mu (x) < \epsilon  \right\} \; . \]
Then \emph{lower and upper quantization scales} of $ \mu$  for a given scaling $ \Scl$ are defined by:
\[ \underline{\Scl}_Q \mu  =  \sup \left\{ \alpha > 0 : \mathcal{Q}_\mu ( \epsilon )   \cdot  \scl_\alpha ( \epsilon ) \xrightarrow[\epsilon \rightarrow 0]{}   + \infty  \right\}   \] and \[ \overline{\Scl}_Q \mu  =  \inf \left\{ \alpha > 0 : \mathcal{Q}_\mu ( \epsilon )    \cdot \scl_\alpha ( \epsilon ) \xrightarrow[\epsilon \rightarrow 0]{}  0  \right\}  \; . \]
Quantization scales of a measure are compared in \cref{resultsC} with box scales of measures:
\begin{definition}[Box scales of a measure] \label{bos scales measure} 
Let $ \Scl$ be a scaling and $ \mu$ a positive Borel measure on  a metric space $(X,d)$. 
 We define the \emph{lower  box scale} and the  \emph{$ * $-lower  box scale}  of $ \mu$ by: 
\[\underline{\Scl}_B \mu = \inf_{ E \in \mathcal{B} } \left\{ \underline{\Scl}_B E : \mu(  E ) >  0 \right\} \qand  \underline{\Scl}_B^*\mu = \inf_{ E \in \mathcal{B} }  \left\{ \underline{\Scl}_B E : \mu( X \backslash E ) = 0 \right\} \; ,   \]
where $\mathcal{B}$ is the set of Borel subsets of $X$. 
Similarly, we define the \emph{upper  box scale} and the  \emph{$ * $-upper  box scale}  of $ \mu$ by: 
\[\overline{\Scl}_B \mu = \inf_{ E \in \mathcal{B} } \left\{ \overline{\Scl}_B E : \mu(  E ) >  0 \right\} \qand  \overline{\Scl}_B^*\mu = \inf_{ E \in \mathcal{B} }  \left\{ \overline{\Scl}_B E : \mu( X \backslash E ) = 0 \right\} \;.  \]
\end{definition}
As for Hausdorff scales of measures, we choose that all box scales of the null measure are equal to $0$ as a convention. 
The following is straightforward:
 \begin{lemma}
\label{quantisation boite}
Let $(X, d)$ be a metric space and $\mu$ a  Borel measure on $X$.  Given  $\Scl$ a scaling, it holds:  
 \[ \underline{\Scl}_Q  \mu   \le \underline{\Scl}^*_{B}  \mu  \qand \overline{\Scl}_{Q} \mu\le \overline{\Scl}^*_{B} \mu \;  . \]
 \end{lemma}
 \begin{proof} We can assume without loss of generality that $ \overline \Scl_B^* \mu $  and $ \underline \Scl_B^* \mu $ are finite.
Let $E$ be a Borel set with total mass such that $ \overline{\Scl}_B E $ is finite, then $ E $ is totally bounded by \cref{fact basic pties}. Now for $\epsilon >0$, consider a minimal covering by $\epsilon$-balls centered at some points $ x_1,  ...,  x_N$ in $ E$.
Since $ \mu (X \backslash E) = 0$, it comes: 
 \[  \int_X d( x,  \left\{ x_i \right\} _{ 1 \le i \le N}) d \mu(x) = \int_E d( x,  \left\{ x_i \right\} _{ 1 \le i \le N}) d \mu(x)  \le  \epsilon < 2 \epsilon \; .  \]
It follows that $ \mathcal{Q}_\mu( 2 \epsilon ) \le \mathcal{N}_\epsilon (E) $, and thus by  \cref{lem inegalite}, we obtain: 
  \[ \underline{\Scl}_Q  \mu   \le \underline{\Scl}_{B} E  \qand  \overline{\Scl}_{Q} \mu \le \overline{\Scl}_{B} E \; .\]
  Since this holds true for any Borel set $E$ with total mass, the desired results come. 
\end{proof}
The following lemma will allow us to compare quantization scales with box scales.  
\begin{lemma}\label{lemer}
Let  $ \mu$  be a Borel measure on $(X,d) $ such that $ \mathcal{Q}_\mu ( \epsilon )  < + \infty $  for any $ \epsilon > 0$. 
Let us fix $ \epsilon > 0 $ and an integer $ N \ge \mathcal{Q}_\mu(\epsilon) $. Thus consider $ x_1, \dots,  x_N \in X $ such that:
\[\int_X  d( x,  \left\{ x_i \right\}_{ 1 \le i \le N} ) d \mu  (x) <  \epsilon  \; . \]
For any $r >0$, with  $E_r := \bigcup_{i = 1}^N  B(x_i, r)$, it holds:  
\[  \mu ( X \backslash E_r)  <   \frac{\epsilon}{r} \; . \]
\end{lemma} 
\begin{proof}
Since $X \backslash E_r $, the complement of $E_r$ in $X$ is the set of points with distance at most $r$  from the set $ \left\{ x_1, \dots, x_n \right\}$, it holds: 
\[   r \cdot \mu (X \backslash  E_r ) \le \int_{X \backslash E_r}  d( x,  \left\{ x_i \right\}_{ 1 \le i \le N} ) d \mu  (x)  < \epsilon \; , \]
which gives the desired result by dividing both sides by $r$. 
\end{proof}
The following result exhibits the relationship between quantization scales and box scales. As far as we know, this result has not yet been proved, even for the specific case of dimension.  It is a key element in the answer to \cref{prob Berger}. 
\begin{theorem}
\label{lem quant boite rep au pb}
Let $ \mu$ be a non null Borel measure on a metric space $(X,d)$. For any scaling $ \Scl$, there exists a Borel set $F$ with positive mass arbitrarily close to $ \mu (X) $ such that:  
\[\underline{\Scl}_{B}  F    \le  \underline{\Scl}_Q  \mu    \qand \overline{\Scl}_{B}  F   \le   \overline{\Scl}_{Q} \mu \; . \]
Consequently, it holds: 
 \[ \underline{\Scl}_{B}  \mu   \le  \underline{\Scl}_Q  \mu   \quad \text{and} \quad \overline{\Scl}_{B}  \mu   \le   \overline{\Scl}_{Q} \mu \; . \]
 \end{theorem}
\begin{proof} 
If  $\mathcal Q_\mu  ( \epsilon ) $ is not finite for some $ \epsilon > 0 $ we have $ \underline{\Scl}_Q \mu = + \infty = \overline{\Scl}_Q \mu$ and nothing to prove. 
So we can assume that $ \mathcal{Q}_\mu ( \epsilon) < + \infty$ for every $ \epsilon > 0 $. We consider two sequences of positive numbers  $ \epsilon_n  := \exp( -n) $ and   ${ r_n  := n^2  \cdot \exp(-n)  = n^2 \cdot  \epsilon_n }$ for $n \ge 1$. 
Then for each $ n \ge 1$, consider a finite set $ C_n \subset X$ with minimal cardinality such that: 
\begin{equation*} \label{epsn quant}
\int_X d (x, C_n) d\mu(x)  < \epsilon_n \; . 
\end{equation*} 
By \cref{lemer}, it holds: 
\begin{equation*} \label{meas E}
\mu(  X \backslash E_n) <   \frac{\epsilon_n}{r_n}  = \frac{1}{n^2}  \quad  \text{ with }  E_n := \bigcup_{ x \in C_n} B ( x , r_n ) \; .  
\end{equation*} 
By Borell-Cantelli lemma, it holds: 
\[ \mu ( \liminf E_n ) = \mu \left( \bigcup_{m \ge 1 } \bigcap_{ n \ge m} E_n \right) = \mu (X) > 0 \; .  \] 
Hence, for sufficiently large $m$, we have $ \mu (F) > 0 $, where $ F := \bigcap_{n \ge m } E_n$. Moreover if $ \mu $ is finite by taking $ m $ even larger we can have $ \mu ( X \backslash F) $ arbitrarily small;  and if $ \mu$ is infinite, we can take $ \mu (F) $ arbitrarily large. 
We now fix such a value of $ m $. 
Observe that $ F \subset E_n $ for every $ n \ge m $. As $ E_n$ is a union of balls of radius $ r_n$, we have trivially: 
\[ \forall n \ge m : \mathcal{N}_{r_n} ( F) \le \Card C_n = \mathcal{Q}_\mu ( \epsilon_n)  \; . \]
Finally by \cref{contdiscr}, and the fact that $ \log r_n \sim \log \epsilon_n \sim \log r_{n+1} $, we obtain: 
\begin{align*}
\underline{\Scl}_B F &= \sup \lbrace \alpha > 0 : \mathcal{N}_{r_n} ( F) \cdot \scl_\alpha ( r_n) \xrightarrow[n \to + \infty]{} + \infty \rbrace \\
&\le  \sup \lbrace \alpha > 0 : \mathcal{Q}_{\mu} ( \epsilon_n) \cdot \scl_\alpha ( \epsilon_n) \xrightarrow[n \to + \infty]{} + \infty \rbrace \\
&= \underline{\Scl}_Q  \mu \; . 
\end{align*} 

Similarly, we can prove $\overline{\Scl}_B F \le \overline{\Scl}_Q \mu$. Then the last two inequalities stated in \cref{lem quant boite rep au pb} follow directly from the definitions of $ \underline{\Scl}_B \mu $ and $ \overline{\Scl}_B \mu $. 
\end{proof}
\subsection{Comparison between local and global scales of measures \label{sec comp q loc} and proof of Theorem C}
To prove \cref{resultsC} we will need: 
\begin{theorem}
\label{rep au pb}
Let $(X,d)$ be a separable metric space and $ \mu$ a finite  Borel measure on $X$. Let $ \Scl$ be a  scaling. It holds: 
\[ \supess \underline{\Scl}_{\loc} \mu  \le \underline \Scl _ Q \mu\quad  \text{and} \quad  \supess \overline{\Scl}_{\loc} \mu  \le \overline \Scl _ Q \mu \; . \]
\end{theorem}
\begin{proof} 
If the lower (respectively upper) local scale of $ \mu$ is zero at almost every point then obviously it is not greater than the lower (respectively upper) quantization scale of $\mu$. Otherwise, consider $ \alpha, \beta > 0 $ such that: 
\begin{equation} \label{ineqsig0}
\alpha  < \supess \underline{\Scl}_{\loc} \mu  \qand \beta < \supess \overline{\Scl}_{\loc} \mu  \; . 
\end{equation} 
Note that $ E: = \left\{ x\in X : \underline{\Scl}_{\loc} \mu (x) > \alpha \qand  \overline{\Scl}_{\loc} \mu (x) > \beta \right\}$ has positive mass. 
Applying \cref{lem quant boite rep au pb} to the restriction of $ \mu $ to $E$ provides a Borel subset $F \subset E$ with $\mu( F)  > 0$ such that:
\[ \underline \Scl_B F \le   \underline \Scl_Q \mu \qand  \overline \Scl_B F \le   \overline \Scl_Q \mu  \; . \]
Yet by \cref{Falconer hausdorff boite} and \cref{Falconer paquet boite}, it holds respectively: 
\begin{equation} \label{ineqsig3}
\Scl_H F \le \underline \Scl_B F \qand \Scl_P F \le \overline \Scl_B F  \; . 
\end{equation}
Now, by setting $\sigma$ as the restriction of  $\mu$ to $F$, we obtain by \cref{infess} and \cref{ineqsig0}:
\begin{equation} \label{ineqsig1}
 \alpha \le \infess \underline \Scl_{\loc} \sigma   \qand   \beta \le \infess \overline \Scl_{\loc} \sigma  \; . 
\end{equation}
By  \cref{lemma dim loc}, it holds:
\begin{equation} \label{ineqsig2}
\infess \underline \Scl_{\loc}\sigma  \le \Scl_H F \qand   \infess \overline \Scl_{\loc}\sigma  \le \Scl_P F \; . 
\end{equation}
Finally, combining \cref{ineqsig3,ineqsig1,ineqsig2} gives: 
\[ \alpha \le \infess \underline \Scl_{ \loc} \sigma \le \Scl_H F \le  \underline \Scl_B F \le \underline \Scl_Q \mu  \; \]
and
\[\beta \le \infess \overline \Scl_{ \loc} \sigma \le \Scl_P F \le  \overline \Scl_B F \le \overline \Scl_Q \mu  \; . \]
Since this holds true for  any $ \alpha$ and $ \beta $ arbitrarily close to  $ \supess \underline{\Scl}_{\loc} \mu $  and  $ \supess \overline{\Scl}_{\loc} \mu $ we obtain the desired result. 
\end{proof}
We are now able to show: 
\begin{proof}[Proof of \cref{resultsC}] \label{lieu resultC}
By \cref{lem quant boite rep au pb}  and \cref{quantisation boite}, it holds:
\[ \underline{\Scl}_{B}  \mu   \le  \underline{\Scl}_Q  \mu   \le \underline{\Scl}^\star_{B}  \mu  \quad \text{and} \quad \overline{\Scl}_{B}  \mu   \le   \overline{\Scl}_{Q} \mu \le  \overline{\Scl}^\star_{B} \mu \; . \] 
By \cref{rep au pb} it holds: 
\[ \supess \underline{\Scl}_{\loc} \mu  \le \underline \Scl _ Q \mu\quad  \text{and} \quad  \supess \overline{\Scl}_{\loc} \mu  \le \overline \Scl _ Q \mu \; . \]
Thus it remains only to show:\begin{equation} \label{eq remaining}
\infess \underline{ \Scl}_{\loc} \mu \le \underline \Scl_B \mu \qand  \infess  \overline{ \Scl}_{\loc} \mu \le  \overline \Scl_B \mu \; .  
\end{equation}  
To prove this, consider a subset $ E \subset X$ with positive mass. Set $ \sigma $ as the restriction of $ \mu $ to $E$. By \cref{infess}, it holds:
\begin{equation} \label{qb1}
\infess  \underline{ \Scl}_{\loc} \mu \le  \infess  \underline{ \Scl}_{\loc} \sigma \qand  \infess  \overline{ \Scl}_{\loc} \mu \le  \infess \overline{ \Scl}_{\loc} \sigma  \; .   
\end{equation} 
By \cref{rep au pb} it holds:
\begin{equation} \label{qb2}
\supess \underline{\Scl}_{\loc} \sigma  \le \underline \Scl _ Q \sigma \quad  \text{and} \quad  \supess \overline{\Scl}_{\loc} \mu  \le \overline \Scl _ Q \sigma \; . 
\end{equation}
Moreover, by \cref{quantisation boite} we have:
\begin{equation} \label{qb3}
 \underline \Scl _ Q \sigma  \le \underline \Scl_B E \qand \overline \Scl _ Q \sigma \le \overline \Scl_B E  \; .
\end{equation}
Combining \cref{qb1,qb2,qb3} provides: 
\[ \infess \underline{ \Scl}_{\loc} \mu  \le \underline \Scl_B E \qand  \infess  \overline{ \Scl}_{\loc} \mu \le  \overline \Scl_B E    \; .   \] 
Taking the infima over such subsets $E\subset X$ with positive mass gives \cref{eq remaining}.  
\end{proof}

\subsection{Proof of \cref{resultsB}}
This subsection contains the proof of \cref{resultsB}, we recall its statement below. 
We will use Vitali covering lemma  to compare local scales with Hausdorff and packing scales. This lemma was first used for the dimensional  case by Tamashiro \cite{tamashiro1995dimensions}. 
\begin{lemma}[Vitali covering lemma]
\label{lemvitali}
Let $(X, d)$ be a separable metric space. Given $\delta >0 $,  $\mathcal{B}$ a family of open balls in $X$ with radii at most $\delta$  and $F $ the union of these balls.   There exists a countable set $ J$ and a $\delta$-pack $ (B(x_j, r_j) )_{j \in J} \subset\mathcal{B}$  of  $F$ such that: 
\[ F \subset \bigcup_j B(x_j, 5r_j) \; . \] 
\end{lemma}
For a proof of this version of the Vitali covering lemma, see  \cite{evans2018measure}[1.5.1, p. 35]. Despite they stated it for the Euclidean case, their proof adapts straightforwardly for separable metric spaces.   
We are now ready to prove: 
\begin{proof}[Proof of \cref{lemma dim loc}]\label{lieu proof 335}
\textsc{Proof of the first inequality:}
If $\supess \underline{\Scl}_{\loc}\mu =0$ or $ \Scl_H X = + \infty$,  it obviously holds $\supess \underline{\Scl}_{\loc}\mu \le \Scl_H X$. 
Otherwise, $ X $ is separable and consider a positive $ \alpha < \supess \underline{\Scl}_{\loc}\mu $. There exists $r_0 >0$ such that the set:
\[ A := \left\{ x \in X : \mu(B(x, r)) \le \scl_\alpha (r),  \ \forall r \in ( 0, r_0 )  \right\}  \; \]
has positive measure. 
Consider $\delta \le r_0$ and a  $\delta$-cover $( B_j )_{j \in J} $ of $A$. It holds:
\[ 0 < \mu(A) \le \sum_{j \in J } \mu( B_j) \le \sum_{j\in J } \scl_\alpha ( \vert   B_j \vert )  \;. \]
Taking the infimum over $ \delta$-covers provides: 
\[ 0 < \mu (A) \le  \mathcal{H}_\delta^{ \scl_\alpha} (A)  \; .  \]
Now as  $\delta$ goes to $0$ we obtain:
\[  0 < \mu (A) \le \mathcal{H}^{ \scl_\alpha} ( A)  \; . \]
It follows: 
\[ \Scl_H X  \ge \Scl_H A   \ge \alpha \; , \]
which allows us to conclude the proof of that first inequality. 

\textsc{Proof of the second inequality:}
If $\supess \overline{\Scl}_{\loc}\mu =0$ or $ \Scl_P X = + \infty$,  it obviously holds $ { \supess \overline{\Scl}_{\loc}\mu \le \Scl_P X}$. Otherwise, $ X $ is separable and consider a positive $ \alpha < \supess \overline{\Scl}_{\loc}\mu $. Put: 
 \[ F = \left\{ x \in X : \overline{\Scl}_{\loc}\mu(x) > \alpha \right\}. \]
 Let  $( F_N)_{ N \ge 1 } $ be a covering of $X$ by Borel subsets.  
Given $ 0 < \beta < \alpha$, by \cref{ineg cons},  there exists  $\delta_0 > 0 $ such that for any $r \le \delta_0$, it holds: 
\begin{equation} \label{const5}
\scl_\alpha ( 5 r) \le \scl_\beta (   r  )  \; . 
\end{equation}
Fix   $\delta \in ( 0,  \delta_0)$ and an integer $N \ge 1$. 
For each  $x$ in $F_N$, by  \cref{contdiscr} there exists a minimal integer $n(x)$ such that: 
\begin{equation} \label{minint}
\mu \left(  B( x, 5 r(x) ) \right)  \le \scl_\alpha( 5 r(x))  \quad \text{ where }  r(x) : = \exp ( - n(x)) \le \delta  \; . 
\end{equation}
We now set: 
 \[ \mathcal{F}  = \left\{ B( x,  r(x)) : x \in F_N \right\} \; .  \]
 Thus by Vitali covering \cref{lemvitali} there exists a countable subset  $J \subset F_N $
  such that   $\left(  B( x ,  r(x) ) \right)_{  x \in J }  $ is a  $\delta$-pack of $F_N$   and $\left(  B( x , 5   r(x) ) \right)_{  x \in J } $  is a covering of $ F_N$.
From there, by \cref{const5,minint} it holds:
\[  \mu(  F_N ) \le \sum_{x \in J} \mu ( B(x, 5 r(x) )) \le \sum_{x \in J} \scl_\alpha( 5  r(x) ) \le\sum_{x \in J} \scl_\beta( r(x)  ) \;. \]
Since this holds true for any $\delta$-pack, we have: 
\[ \mathcal{P}_{\delta}^{\scl_\beta}(F_N) \ge \mu(F_N) \; ,   \]
and then taking $\delta$ arbitrarily  close to  $0$ leads to:
\[\mathcal{P}_{0}^{\scl_\beta}(F_N) \ge \mu(F_N)   \; .  \]
Summing over $N \ge 1$ provides: 
\[    \sum_{N \ge 1} \mathcal{P}_{0}^{\scl_\beta}(F_N) \ge \sum_{N \ge 1}  \mu(F_N) \ge \mu(F) > 0 \; . \]
Recall that $(F_N)_{ N \ge 1}$ is an arbitrary covering by Borel sets of $F$, thus by definition of the packing measure and the latter equation, it holds: 
\[ \mathcal{P}^{\scl_\beta} ( F) \ge \mu(F) > 0 \;. \]
It holds then  $  \Scl_P F \ge \beta$ for any $ \beta < \alpha <  \supess \overline \Scl_\loc \mu$, which allows us to conclude the proof.
\end{proof}
We deduce then: 
\begin{proposition} \label{B first}
Let $(X, d) $ be a metric space and  $\mu$ a Borel measure on $X$,  then: 
\[    \infess \underline{\Scl}_{\loc}\mu  \le   \Scl_H \mu \qand   \infess \overline{\Scl}_{\loc} \mu  \le \Scl_P \mu   \; ,\]
and 
\[ \supess \underline{\Scl}_{\loc}\mu \le  \Scl^*_H \mu    \qand     \supess \overline{\Scl}_{\loc} \mu \le    \Scl^*_P \mu   \; . \]
\end{proposition}
\begin{proof}
The second line of inequalities is given by \cref{rem 35}. It remains to show the first line of inequalities. 
Let $ E $ be a Borel subset of $X$ with $ \mu$-positive mass. With $ \sigma $ the restriction of $ \mu$ to $E$, it holds by \cref{lemma dim loc}:
\[ \supess \underline \Scl_{\loc}  \sigma \le \Scl_H  E   \qand \supess \overline \Scl_{\loc}  \sigma \le \Scl_P E \;  . \]
Then, by \cref{infess},it follows: 
\[ \infess \underline \Scl_{\loc}  \mu \le \Scl_H  E   \qand \infess \overline \Scl_{\loc}  \mu \le \Scl_P E \;  .  \]
Taking the infima over $E $ with positive mass ends the proof. 
\end{proof}
 Explicit links between packing scales, Hausdorff scales and local scales of measures can now be established by proving \cref{resultsB}. Let us first recall its statement:
Let $(X, d) $ be a metric space and  $\mu$ a Borel measure on $X$,  then it holds: 
\[  \Scl_H \mu = \infess \underline{\Scl}_{\loc}\mu   \le     \Scl_P \mu = \infess \overline{\Scl}_{\loc} \mu \]
and 
\[  \Scl^*_H \mu = \supess \underline{\Scl}_{\loc}\mu  \le    \Scl^*_P \mu = \supess \overline{\Scl}_{\loc} \mu  \; . \]
\begin{proof}[Proof of \cref{resultsB}]
By \cref{B first} it remains only to show four inequalities: 
\[ \Scl_H \mu  \underbrace{\le}_{(i)}   \infess \underline{\Scl}_{\loc} \mu \quad ; \quad   \Scl_P \mu  \underbrace{\le}_{(ii)}   \infess \overline{\Scl}_{\loc} \mu  \; \]
and  
\[ \Scl^*_H \mu  \underbrace{\le}_{(iii)}   \supess \underline{\Scl}_{\loc} \mu   \quad ; \quad\Scl^*_P \mu  \underbrace{\le}_{(iv)}   \supess \overline{\Scl}_{\loc} \mu     \; . \]
Note that for each of the above inequalities, if the right-hand side quantity is infinite, there is nothing to prove. In each of the following proofs, we will assume that this quantity is finite. \newline
\textsc{Proof of $(i)$:} 
Fix  $ \alpha > \infess \underline{\Scl}_{\loc} $ and $\beta > \alpha$. By definition of scaling,  there exists $\delta > 0$ such that for any  $r \in (0,  \delta) $ it holds 
\begin{equation} \label{i_ineq5}
\scl_\beta( 5r) \le \scl_\alpha( r) \; . 
\end{equation}
Denote: 
 \[  F := \{ x \in X :  \underline\Scl_{\loc} \mu (x) < \alpha \}  \; . \]
 Then it holds  $\mu( F) >0$  and by \cref{contdiscr} for any $x$ in $F$ there exists a minimal integer $n(x)$  such that:  
 \begin{equation} \label{i_mesball}
  \mu \left(  B( x, r(x))\right)  \ge \scl_\alpha( r(x))  \quad \text{ where }   r(x) : = \exp ( - n(x)) \le  \delta  \; .
 \end{equation}
Now set: 
 \[ \mathcal{F}  := \left\{ B( x,  r(x)) : x \in F \right \}. \]
By Vitali covering \cref{lemvitali}, there exists a countable subset  $J \subset F $ such that  $\left(  B( x ,  r(x) ) \right)_{  x \in J } $ is a $ \delta$-pack of $F$ and $F \subset \bigcup_{x \in J} B( x, 5 r(x))$. 
Then, by \cref{i_ineq5,i_mesball} it holds: 
 \[ \sum_{x \in J} \scl_\beta (  5r(x)) \le   \sum_{x \in J} \scl_\alpha ( r(x) )  \le   \sum_{x \in J} \mu  ( B(x, r(x) ) )  \le \mu (F ) \; .  \]
It follows that $\mathcal{H}_\delta^{\scl_\beta} ( F)  \le \mu (F)$ and thus $  \mathcal{H}^{\scl_\beta} ( F)  \le \mu (F)$ as $ \delta$ goes to $ 0$. It follows that $ \Scl_H \mu \le \Scl_H F \le \beta$. Taking $ \beta > \alpha$ close to $\infess \underline{\Scl}_{\loc} \mu $,  we  obtain inequality $(i)$.  \newline
\textsc{Proof of $(ii)$:}
Consider $ \alpha > \supess \overline{\Scl}_{\loc} \mu $. Set $ E = \left\{ x \in X : \overline \Scl_{\loc} \mu (x) < \alpha\right\}$. Observe that $\mu(X \backslash E) = 0$ and:  
\[ E = \bigcup_{i \ge1 } E_i  \quad \text{where} \ E_i := \left\{ x \in E : \forall r \le 2^{-i},  \ \mu( B(x,  r) )   \ge \scl_\alpha(r) \right\}  \; .\]
By $\sigma$-stability of packing scales provided by \cref{countstab}, it holds $ \Scl_P E = \sup_{i \ge 1}  \Scl_P E_i$. It is then enough to show that for any $i \ge 1$,  we have $  \Scl_P E_i \le \alpha$.  Indeed, then  taking $\alpha$ arbitrarily close to  $\supess \overline{\Scl}_{\loc} \mu$ allows us to conclude. In that way, let us fix  $i \ge 1$. 
Fix  also $\delta \in ( 0,  2^{ -i} ) $ and  consider a countable set $ J$ and $ ( B_j )_{j \in J}$ a $\delta$-pack of $E_i$. 
It follows: 
\[ \sum_{ j \in J}  \scl_\alpha( \vert B_j \vert )  \le   \sum_{ j \in J}  \mu(B_j  )  \le 1  \; . \]
Since this holds true for any $\delta$-pack, we have: 
\[ \mathcal{P}_\delta^{\scl_\alpha}( E_i) \le 1 \; . \]
As  $\delta$ goes to $0$ we obtain:
 \[ \mathcal{P}^{\scl_\alpha}( E_i)   \le 1 \; .\]
It follows that  we indeed have $  \Scl_P E_i \le \alpha$. \newline 
\textsc{Proof of $ (iii)$:}
Fix $ \alpha > \supess \underline{\Scl}_{\loc} $. For $\beta > \alpha$, consider  $\delta > 0$ such that for any $r \in (0,  \delta) $ it holds:
\begin{equation} \label{iii_ineq5} 
\scl_\beta( 5r) \le \scl_\alpha( r) \; . 
\end{equation}
Denote: 
 \[  F := \left\{ x \in X :  \underline\Scl_{loc} \mu (x) < \alpha \right\}  \; .  \]    
Note that $ F $  has total mass and by \cref{contdiscr} for every $x \in F$ there exists a minimal integer $n(x)$ such that:
\begin{equation} \label{iii_mesball}
   \mu \left(  B( x, r(x) )\right)  \ge \scl_\alpha( r(x)) \quad \text{ where } r(x) : = \exp ( - n(x)) \le \delta  \; .
\end{equation}

Now put: 
 \[ \mathcal{F}  := \left\{ B( x,  r(x)) : x \in F \right\} \; .\]
By Vitali's \cref{lemvitali}, there exists a countable subset $J \subset F $ such that $\left(  B( x ,  r(x) ) \right)_{  x \in J }  $ is a $\delta$-pack and $F = \bigcup_{ x\in J} B(x, 5 r(x) )    $. 
Thus, by \cref{iii_ineq5,iii_mesball} it holds:
 \[ \sum_{x \in J} \scl_\beta (  5 r(x)  ) \le   \sum_{x \in J} \scl_\alpha ( r(x) )  \le   \sum_{x \in J} \mu  ( B(x, r(x)) ) \le \mu (F )  \; . \]
It follows that $\mathcal{H}_\delta^{\scl_\beta} ( F)  \le \mu (F)$ and thus $  \mathcal{H}^{\scl_\beta} ( F)  \le \mu (F)$ as $ \delta$ goes to $ 0$. It follows that $ \Scl^*_H \mu \le \Scl_H F \le \beta$. Taking $ \beta > \alpha$ close to $\supess \underline{\Scl}_{\loc} \mu $,  we  obtain inequality $(iii)$. \newline
\textsc{Proof of $(iv)$:}
Fix $ \alpha > \infess   \overline \Scl_{\loc} \mu$. Consider the set $ {E := \left\{ x \in X : \overline \Scl_{\loc} \mu < \alpha \right\} }$.  Then observe that $\mu(X \backslash E) = 0$  and that we can write:
\[ E = \bigcup_{i \ge1 } E_i  \hspace{1.5cm} \text{where} \ E_i = \left\{ x \in E : \forall r \le 2^{-i},  \ \mu( B(x,  r) )   \ge \scl_\alpha(r) \right\} \; . \] 
By \cref{countstab}, we have $ \Scl_P E = \sup_{i \ge 1}  \Scl_P E_i$,  it is then enough to show that for any $i \ge 1$, we have $  \Scl_P E_i \le \alpha$. We indeed can take  $\alpha$ arbitrarily close to $\infess \overline{\Scl}_{\loc} \mu$. We then fix $i \ge 1$. Fix $\delta \in ( 0,  2^{ -i} ) $.  We consider $J$ a countable set and $ (B_j )_{j \in J}$ a $\delta$-pack of $E_i$. 
Then it holds:
\[ \sum_{ j \in J}  \scl_\alpha( \vert B_j \vert )  \le   \sum_{ j \in J}  \mu(B_j  )  \le 1 \; . \]
Since this holds true for any $\delta$-pack, it follows: 
\[ \mathcal{P}_\delta^{\scl_\alpha}( E_i) \le 1 \; .\]
When $\delta$ tends to $0$, the latter inequality leads to:
 \[ \mathcal{P}^{\scl_\alpha}( E_i)  \le \mathcal{P}_0^{\scl_\alpha}( E_i) \le 1 \; .\]
From there, we deduce $  \Scl_P E_i \le \alpha$,  which concludes the proof of $ \Scl_P^* \mu  \le  \supess \overline{\Scl}_{\loc} \mu $ and thus the one of \cref{resultsB}. 
\end{proof}
\section{Examples and applications of theorems of comparison of scales}
\subsection{Scales of infinite products of finite sets \label{scales of group} }
Natural toy models in the study of scales are given by  products $ Z = \prod_{ n \ge 1 } Z_k $ of finite sets. 
To define the metric $ \delta$ on this set, we fix a sequence $ ( \epsilon_n)_{n  \ge 1}$  decreasing to $ 0$ and verifying $  \log \epsilon_{n+1} \sim \log  \epsilon_n $ as $ n \to + \infty$. We then define the distance between  $ \underline{x}  = ( x_n)_{n \ge 1}  \in Z $  and $\underline{x}'  = ( x'_n)_{n \ge 1}  \in Z $ by: 
\[ \delta ( \underline{x}  , \underline{x}' ) := \epsilon_m  \; ,  \]
where $ m = \nu(\underline{x}  , \underline{x}' ) := \inf \lbrace n \ge 1 : x_n \neq x'_n \rbrace  $ is the minimal index such that the sequences $ \underline{x} $ and $ \underline{x}'$ differ. 
We endow each $Z_n$ with the discrete topology, thus $ \delta$ provides the product topology on $Z$. 

A natural measure on $Z$ is the following product measure: 
\[ \mu := \otimes_{n \ge 1}  \mu_n \; ,  \]
where $ \mu_n$ is the equidistributed measure on $ Z _n$ for $ n \ge 1$. 
Note that the metric is highly dependent on the choice of the sequence $ ( \epsilon_n)_{n \ge 1} $. We will tune both values of the sequences $( \epsilon_n)_{ n \ge 1} $ and $ (\Card Z_n)_{ n \ge 1} $ to reach different values of scales for $ Z $ and $\mu$. 
First, box scales of $ Z$ coincide with local scales of $ \mu$ according to the following: 
\begin{proposition}
\label{toy example}
For any scaling $ \Scl$, it holds for every $\underline{x}  \in Z$: 
\begin{equation} \label{eq ptoy1}
\underline{\Scl}_{\loc} \mu (\underline{x} )  = \underline \Scl_B Z = \sup \left\{ \alpha > 0 : \scl_\alpha ( \epsilon_{n}) \cdot \prod_{k =1}^n  \Card Z_k    \xrightarrow[n \rightarrow+ \infty]{} + \infty \right\} \; 
\end{equation} 
and 
\begin{equation} \label{eq ptoy2} \overline{\Scl}_{\loc} \mu (\underline{x} )  = \overline \Scl_B Z = \inf \left\{ \alpha > 0 : \scl_\alpha ( \epsilon_{n}) \cdot \prod_{k =1}^n \Card Z_k  \xrightarrow[n \rightarrow  + \infty]{} 0 \right\} \; .
\end{equation} 
\end{proposition}
\begin{proof} 
To prove this, first note that for $ \underline{x} \in X $ and $ n \ge 1 $, it holds: 
\[  B  ( \underline{x}  , \epsilon_n) = \lbrace \underline{x}' \in Z : x'_1 = x_1 , \dots , x'_n = x_ n \rbrace  \; .  \]
Such a ball is usually called a $ n$-cylinder.
By definition of $ \mu$ it holds:
\[ \mu  \left( B( \underline{x} , \epsilon_n )  \right) =  \prod_{ k = 1 }^n \mu_k ( x_k)   = \prod_{ k = 1 }^n  \frac{1 }{\Card Z_k} \; .  \]  
Moreover as two balls with the same radius are either equal or disjoint, there are exactly  $ \prod_{ k = 1} ^n \Card Z_k $ different balls of radius $ \epsilon_n$ and in particular  $ \mathcal{N}_{ \epsilon_n}( Z )  = \prod_{ k = 1} ^n \Card Z_k$.

We just proved: 
\begin{equation} \label{eq toy}
\mathcal{N}_{\epsilon_n} (Z)  =   \prod_{k = 1 }^n \Card Z_k   =  \mu ( B ( \underline z , \epsilon_n))^{-1}  \; . 
\end{equation}  
Since $ \log \epsilon_{n+1} \sim \log \epsilon_n $ as $ n \to  + \infty$,  by \cref{contdiscr} and  \cref{eq toy}, we obtain: 
\begin{equation} \label{eqtoy3} 
\underline{\Scl}_B Z = \sup \left\{ \alpha > 0 :  \prod_{ k = 1 } ^n \Card Z_k \cdot \scl_\alpha ( \epsilon_n) \xrightarrow[n \to + \infty ]{} 0 \right\} = \underline{\Scl}_{\loc}  \mu ( \underline{x} ) \; , 
\end{equation} 
which is actually \cref{eq ptoy1}. We obtain similarly \cref{eq ptoy2}. 

\end{proof} 

 \cref{toy example} together with \cref{resultsA} and \ref{resultsC} directly implies: 
\begin{corollary}
\label{toy example cor}
We have moreover: 
\[  \Scl_H Z = \underline{\Scl}_Q  \mu = \underline \Scl_B Z   \qand  \Scl_P Z= \overline{\Scl}_Q  \mu   = \overline \Scl_B Z  \; . \]
\end{corollary}

Also, for $ \Scl = \ord$ and if $  \frac{-\log \epsilon_{n}}{n}  $ is moreover converging to some positive finite constant, we have: 
\begin{corollary}
\label{toy example order}
Suppose further that $ \frac{-\log \epsilon_{n}}{n} $ converges to $C> 0 $  when $ n \rightarrow  + \infty$ ( a typical choice is  $\epsilon_{n} = \exp ( C \cdot n) $). 
Then for any scaling $ \Scl$, it holds: 
\[  \ord_H Z = \underline \ord_B Z = \liminf_{n \rightarrow + \infty}\frac{1 }{C  \cdot  n } \log  \left(  \sum_{k = 1 }^n \log   \left(  \Card Z_k  \right)  \right)   \;  \]
and 
\[ \ord_P Z = \overline \ord_B Z = \limsup_{n \rightarrow + \infty}\frac{1 }{C  \cdot  n } \log \left(  \sum_{k = 1 }^n \log \left(  \Card Z_k  \right)  \right)  \; .  \]
\end{corollary}
\begin{proof}
Recall that the scaling defining order is given for $ \alpha > 0   $ and $ \epsilon \in (0,1)$ by $ \scl_{\alpha}^{2,1} = \exp ( - \epsilon^{-\alpha} ) $. The first equality in \cref{eqtoy3} reads as: 
 \[ \underline \ord _B ( X ) =  \sup \left\{ \alpha > 0 :  - \epsilon_n^{-\alpha } +  \sum_{ k = 1 } ^n \log  \Card Z_k   \xrightarrow[n \to + \infty ]{} + \infty \right\}   \;  .   \]
 As  $  \epsilon_n^{-\alpha}  = e^{- \alpha \log \epsilon_n} $ and $ - \log \epsilon_n \sim Cn $, we obtain the first announced formula. The second formula is deduced similarly.  
\end{proof}

Such examples of products of finite sets allow us to exhibit compact metric spaces with arbitrarily high orders : 
\begin{example} \label{ex inho}
For any  $ \alpha \ge \beta > 0 $, there exists a compact metric probability space $(Z,\delta, \mu )$ such that for any $ z \in Z $:  
\[ \beta =  \underline{\ord} _{\loc}  \mu (z)  = \ord_H Z =  \underline{\ord} _{Q}  \mu  = \underline{\ord} _B   Z \]
and
\[ \alpha  =  \overline{\ord} _{\loc}  \mu (z)  = \ord_P Z =  \overline{\ord} _{Q}  \mu   = \overline{\ord} _B  Z  \; . \]
\end{example}
In particular, with $ \alpha > \beta$  we obtain examples of metric spaces with finite order such that the Hausdorff and packing orders do not coincide. Moreover, for a countable dense subset $ F $ of $X$, it holds:
\[  \ord_H F = \ord_P F = 0  \qand  \underline \ord_B F = \beta < \alpha = \overline \ord_B F  \; . \]
It follows that none of the inequalities of \cref{resultsA} for the case of order is an equality in the general case.
\begin{proof}[Construction of \cref{ex inho}] 
Let  $(u_k)_{k \ge 0 } $  be the sequence defined by: 
\begin{equation*}
u_k =
\left\lbrace
\begin{array}{ccc}
\lceil \exp( \exp (\beta \cdot k))  \rceil & \mbox{if} & c^{2j}  \le k <  c^{2j +1 }   \\
\lfloor \exp( \exp (\alpha \cdot k) )  \rfloor  & \mbox{if} &   c^{2j+1}  \le k <  c^{2j +2  }
\end{array}\right.
\end{equation*}
where $c = \lfloor \frac{\alpha}{ \beta} \rfloor + 1$.
We denote $ Z := \prod_{n \ge 1}  \Z / u_k \Z  $  endowed with the metric $ \delta $ defined by: 
\[ \delta  ( \underline z , \underline w ) :=  \exp ( - \inf\left\{ n \ge 1 : z_n \neq w_n\right\} )  \]
for $ \underline z  = (z_n ) _{n \ge 1 }$ and $ \underline w  = (w_n)_{n \ge 1}$ in $Z$. 
Let us denote $ \lambda_n = \frac{1}{n} \log  \sum_{k=1}^n \log u_k $. 
Thus by \cref{toy example order}, it follows: 
  \[ \ord_ H Z = \underline \ord_B Z = \liminf_{n \to  + \infty} \lambda_n  \qand \ord_ P Z = \overline \ord_B Z = \limsup_{n \to +  \infty} \lambda_n \; . \]
  It remains to show that $  \lambda^- := \liminf_{ n \to  + \infty }  \lambda_n = \beta $ and $  \lambda^+ := \limsup_{ n \to  + \infty }  \lambda_n = \alpha $ in order to show that $ (Z,\delta)$ satisfies the desired properties. First notice that $ \exp (\exp ( \beta  \cdot n ) )\le u_n \le \exp ( \exp ( \alpha \cdot n)) $ for every integer $ n $. It follows that $ \lambda^- \ge \beta $ and $ \lambda^+ \le \alpha$. 
Denote $ n_j = c^{2j +1 }$ and observe that:
\[ \lambda_{n_j}   \ge \frac{1}{n}  \log \log (u_{n_j}  )  \xrightarrow[ j \to + \infty]{} \alpha \; .  \]
Thus we have $ \lambda^+ \ge \alpha$. 
Moreover, denote  $ m_j = c^{2j+1} - 1 $. We have the following: 
\begin{lemma}
For any $ j \ge 1$ and for any $  1 \le  k \le m_j$, it holds: 
\[  u_k  \le  u_{m_j}  \; . \]
\end{lemma}

\begin{proof} 
If  $ c^{ 2j } \le k \le  m_j$,  then  $u_k = \lceil \exp( \exp (\beta \cdot k)) \rceil  \le \lceil \exp( \exp (\beta \cdot m_j)) \rceil  =  u_{m_j} $. Otherwise, we have $k < c^{2j}$, and then $u_k <  \lfloor \exp( \exp (\alpha \cdot c^{2j} ))  \rfloor <   \lceil \exp( \exp (\beta \cdot c^{2j+1} ))  \rceil = u_{m_j} $, since $ \alpha <  \beta  \cdot c$. 
\end{proof} 
From the above lemma, we have:  
\[  \lambda_{m_j} \le \frac{1}{m_j} \log m_j \log ( u_{m_j} )  \xrightarrow[j \rightarrow +  \infty]{}   \beta  \; ,   \]
 and so $ \lambda^- \le \beta$ which ends the construction of \cref{ex inho}.   
\end{proof}
\subsection{Functional spaces \label{proof func}}
This last sub-section consists in showing \cref{lem diff} that allows us to show \cref{KT+}. 
\medskip
 Denote by $ \| \cdot \|_{ C^k} $ the $C^k$-uniform norm on $ C^k ( [0,1]^d , \R) $: 
\[ \| f \|_{ C^k} :=   \sup_{  0 \le j \le k }  \| D^j f \|_{\infty} \;  .  \]
\begin{definition}
 For integers $ d \ge 1 $ and $ k \ge 0  $, let us denote: 
\[ \mathcal F^{d,k,0} :=  \left\{ f \in C^k ( [0,1]^d, [-1,1]) : \| f \|_{C^k} \le 1  \right\} \;   \] 
the $ C^k$-unit ball; and for $ \alpha \in (0,1]$ let us define:
\[ \mathcal F^{d,k,\alpha} :=  \left\{ f \in C^k ( [0,1]^d, [-1,1]) : \| f \|_{C^k} \le 1   \text{ and $D^k f$ is $\alpha$-H{\"o}lder with constant $1$  } \right\} \; .  \] 
Recall that for  $ \alpha>0$, the map $D^kf$ is $ \alpha$-H{\"o}lder with constant  $1$ if for any $x,y \in [0,1]^d$ it holds: 
\[ \| D^k f(x) - D^k f(y) \|_{\infty} \le \| x - y \|^\alpha \; . \]
\end{definition} 
Let us recall the  asymptotic given by Kolmogorov-Tikhomirov {\cite{tikhomirov1993varepsilon}[Thm XV]} on the covering number of $ { ( \mathcal{F}^{d,k,\alpha}, \| \cdot \|_\infty ) }$  stated in \cref{KT}:   
\[ C_1  \cdot \epsilon^{-\frac{d}{ k + \alpha } }  \ge   \log \mathcal{N}_\epsilon ( \mathcal{F}^{d, k, \alpha} )  \ge  C_2  \cdot  \epsilon^{-\frac{d}{ k + \alpha } } \; ,    \]
where $ C_1 > C_2  > 0$ are two constants depending on $  d,k$ and $ \alpha$.
In order to prove \cref{KT+} which states that box, packing and Hausdorff scales of $ \mathcal{F}^{d,k,\alpha}$ are all equal to $ \frac{d}{k+ \alpha}$, by \cref{resultsB}, it remains to prove \cref{lem diff}. 
The latter states: 
\[ \ord_H \mathcal{F}^{d,k, \alpha} \ge \frac{d}{k + \alpha} \; .\]
\begin{proof}[Proof of \cref{lem diff}] \label{proof of lem diff}
We first assume $ \alpha > 0$. The case $ \alpha = 0$ will be deduced at the end of the proof. 
We consider the following product of finite sets: 
\[   \Lambda  = \prod_{n \ge 1}  \left\{ 0 , 1  \right\}^{ \mathcal{R}_n }  \; .  \]
where: 
\begin{equation}  
\mathcal{R}_n :=  \left\{ \left(  \frac{i_1 }{R^n}, \dots,  \frac{i_d }{ R^n } \right)  : i_1, \dots, i_d \in \left\{ 0, \dots, R^{n}-1 \right\}  \right\} \text{ with } R := \lfloor 3^{\frac{1}{k + \alpha } } \rfloor + 1   \; .  
\end{equation}
Observe that $ \mathcal{R}_n$ is a meshgrid of step $R^{-n}$ of $[ 0, 1 ]^d$ with cardinal $ R^{nd}$. 
   We endow $\Lambda$ with the ultrametric distance  $ \delta $ defined by:
\[ \delta ( \underline \lambda, \underline \lambda'   ) = \epsilon_m  \; ,  \]
with $ m$ the minimal index such that the sequences $ \underline \lambda$ and  $ \underline \lambda'$  differ and  $ ( \epsilon _n)_{n \ge 1} $  a decreasing sequence of positive real numbers. The following will enable us to conclude the proof of \cref{lem diff}: 
\begin{lemma} \label{bilip} 
We can choose the sequence $( \epsilon_n)_{n \ge 1}$  such that:
\begin{equation} \label{eq eps} 
 - \log \epsilon_n \sim   n \cdot\log R^{ k+\alpha}  \text{ as } n \to + \infty \; . 
\end{equation}
and such that there exists an expanding embedding $ I : ( \Lambda , \delta) \to  (\mathcal{F}^{d,k,\alpha}, \| \cdot \|_\infty ) $. More precisely, for every $ \lambda  , \lambda' \in \Lambda $ it holds: 
\begin{equation} \label{eq exp}
\| I ( \lambda ) - I (\lambda' ) \|_\infty  \ge   \frac{1}{2} \delta ( \lambda, \lambda' )   \; .
\end{equation}
\end{lemma}
The choice of $  (\epsilon_n) _{n \ge1} $  will be given in \cref{eps choice}. 
\cref{bilip} is proven below. Let us see how it allows us to finish the proof of \cref{lem diff}. Since $ \log \epsilon_{n+1} \sim \log \epsilon_n $, by \cref{toy example order} it holds: 
\[ \ord_H \Lambda  = \liminf_{n \rightarrow  + \infty}\frac{1}{   n \log R^{k + \alpha} } \log \left(  \sum_{j=1}^n  \log \Card \left( \left\{ 0 , 1  \right\}^{ \mathcal{R}_j } \right) \right)   = \liminf_{n \rightarrow  + \infty}\frac{1}{   n \log R^{k + \alpha} } \log \left(  \sum_{j=1}^n   R^{dj} \cdot \log 2 \right)  = \frac{d }{ k +  \alpha} \; . \]
Now  \cref{bilip} together with \cref{cor dil} implies: 
\[ \ord_H \mathcal{F}^{d,k,\alpha}   \ge   \ord_H  \Lambda =  \frac{d}{k + \alpha}  \; ,  \]
which is the desired inequality from \cref{lem diff}. \newline
We now show  \cref{bilip} using two intermediary results \cref{estime f nq,bound}. We will actually deduce \cref{bound} from \cref{estime f nq}. 
\begin{proof}[Proof of \cref{bilip}]
Let us denote $ q := k+ \alpha $ so that $ R = \lfloor 3^{ \frac{1}{q} } \rfloor + 1 $. 
For $ f \in \mathcal{F}^{d,k,\alpha} $, let $ \| f\|_{q} $  be the $ \alpha$-H{\"o}lder constant of $ D^kf$, i.e the minimal constant $C>0$ such that for any $x,y \in [0,1]^d $ it holds: 
\[ \|D^k f(x)-D^k f(y)\|_\infty \le C \cdot \|x-y\|^\alpha \; .\]
Note that $ \| \cdot \|_q$ is a semi-norm on  $\mathcal{F}^{d,k, \alpha}$ and moreover we have: 
\[ \mathcal{F}^{d,k,\alpha}= \left\{ f \in \mathcal{F}^{d,k, 0 } : \| f \|_{q } \le 1   \right\} \; . \] 
We consider the following function defined on $ \R$: 
\[ \phi(t) := 4^q \cdot    t^q \cdot ( 1-t)^q \cdot 1\!\! 1 _{0<t<1} \; .   \]
It is clear that $\supp \phi = [0,1]$, $ \phi (0) = 0 = \phi (1) $,  $   \| \phi \|_\infty = \phi(1/2) = 1$. Observe also that $ D^k \phi ( 0 ) = 0 = D^k \phi (1) $. \newline 
Let us denote:
\[ \Phi : x \in \R^d \mapsto \phi ( 2 \| x \| )  \; .\]
Note that $ \Phi  $ is of class $ C^k$ with support $  ( \| x \| \le 1/2) $ and $  \| \Phi \|_\infty =  1$. Also  $ D^k \Phi$ equals $ 0$ outside $ ( \| x \| < 1/2 )$ and  is $ \alpha$-H\"older with constant: 
  \begin{equation} \label{Phiq}
0 <  \|\Phi \|_q < + \infty \; . 
\end{equation} 
We now proceed to the construction of the embedding. We first fix $ n$ and denote $ \mathcal{R} :=  \mathcal{R}_n$. To each $ \lambda = ( \lambda_r )_{ r \in \mathcal{R} }  \in  \left\{ 0, 1 \right\}^{\mathcal{R}}$ we associate the following map:
 \[ f_{\lambda} : x  \in [0,1]^d  \mapsto \epsilon_n \cdot\sum_{r \in \mathcal{R}} \lambda_r \cdot   \Phi (   R^n (x  - r )  )  \;,   \]
with: 
\begin{equation} \label{eps choice} 
\epsilon_n := \frac{3}{\pi^2\cdot n^2\cdot  R^{qn} \cdot  \|\Phi\|_{q}} \;.
\end{equation}
First note that $ (\epsilon_n)_{n \ge 1} $ obviously verifies \cref{eq eps}. 
We have the following result: 
\begin{lemma} \label{estime f nq}
For every $ \lambda \neq \lambda'$  in $ \lbrace 0 ,1  \rbrace ^ \mathcal{R}$, it holds: 
 \[  \|f_\lambda - f_{\lambda'} \|_\infty  = \epsilon_n     \qand \|f_\lambda  -  f_{\lambda'}   \|_q  \le    \frac{6
}{\pi^2 \cdot n^2 }  \; .\] 
\end{lemma}
\begin{proof}
For any $x \in [0,1]^d$, there exists at most one point $ r \in \mathcal{R} $ such that $ \| x - r \| <   R^{-n} /2  $. Consequently, the maps $ x \mapsto \Phi (R^n( x - r ) ) $ for $ r \in \mathcal{R}$ have  supports with disjoint interiors. It comes then:   
 \[\|  f_\lambda - f_{\lambda'}\|_\infty  = \epsilon_n \cdot   \left\| \sum_{r \in \mathcal{R} } ( \lambda_r - \lambda'_r)\cdot \Phi ( R^n ( \cdot - r  )  ) \right\|_\infty   = \epsilon_n \cdot \sup_{r \in \mathcal{R}}  \vert \lambda_r -\lambda'_r \vert \cdot \| \Phi \|_\infty = \epsilon_n  \; . \]
This provides the equality of \cref{estime f nq}. 
 
Now note similarly  that the maps $ x \mapsto D^k \Phi (R^n \cdot ( x - r ) ) $ for $ r \in \mathcal{R}$  also have supports with disjoint interiors and moreover they  are  equal to $ 0$  on the boundary of their support. Thus for every $x \in [ 0,1 ]^d $ there exists at most one value of  $ r \in \mathcal{R}$ such that $ D^k \Phi ( R^n ( x - r ))  $ is not null. 
It follows:  
\begin{align*}
\| f_\lambda - f_{\lambda'}  \|_{q} &= \epsilon_n  \cdot \left\| \sum_{r \in \mathcal{R}}   ( \lambda_r -\lambda'_r) \cdot   \Phi ( R^n ( \cdot - r) ) \right\|_q  \\ 
 &\le \epsilon_n \cdot  \sup_{  r , s \in \mathcal{R}}   \  \sup_{x \neq y \in \R^d}  \frac{\left\| ( \lambda_r -  \lambda_r') \cdot  R^{kn}  D^k \Phi ( R^n ( x - r) ) - ( \lambda_s  - \lambda'_s)  \cdot  R^{kn} D^k  \Phi ( R^n  ( y - s) )  \right\|_\infty}{ \| x-y \|^{\alpha} } \\ 
 &= \epsilon_n \cdot R^{nq} \cdot  \sup_{  r , s \in \mathcal{R}}   \  \sup_{X \neq Y \in \R^d}  \frac{\left\| ( \lambda_r -  \lambda_r') \cdot   D^k \Phi (X- R^n  r)  - ( \lambda_s  - \lambda'_s)  \cdot  D^k  \Phi ( Y - R^n s )  \right\|_\infty}{ \| X-Y \|^{\alpha} }
 \; 
\end{align*}
where the last equality is obtained using the change of variable $ X = R^n x $ and $ Y = R^n y$.
Using the triangular inequality of $ \| \cdot \|_q$, it follows: 
\[ \| f_\lambda - f_{\lambda'}  \|_{q}  \le 2 \epsilon_n \cdot  R^{qn} \cdot  \sup_{ r \in \mathcal{R} } \| \Phi  (  \cdot - R^n r ) ) \|_q  = 2 \epsilon_n \cdot  R^{qn} \| \Phi \|_q   \; .   \]
This allows us to conclude as the latter term is equal to $ \frac{6}{\pi^2 n^2}  $ by \cref{eps choice}. 

\end{proof}
Note in particular, in \cref{estime f nq},  that if $\lambda' = 0 $ then $f_{\lambda'} = 0 $, thus for $ \lambda \neq 0$, it holds:
\begin{equation} \label{norm f}
   \| f_\lambda \|_\infty  = \epsilon_n \qand \|f_\lambda  \|_q  \le    \frac{6
}{\pi^2 \cdot n^2 }  \;  .
\end{equation} 
To embed $ \Lambda$ into $ \mathcal{F}^{d,k,\alpha}$, we use the following that we will prove using the above \cref{estime f nq}: 
\begin{lemma} \label{bound}
For every  $  \underline \lambda = (\lambda_n )_{ n \ge 1 }  \in \Lambda $, the function series $ \sum_{n \ge 1 }  f_{\lambda_n}  $ converges in $ C^0 ( [ 0,1]^d, [-1,1] ) $ and moreover its limit lies in $ \mathcal{F}^{d,k,\alpha} $.
\end{lemma}
\begin{proof} 
By the equality in \cref{norm f} the normal convergence of the series $ \sum f_{\lambda_n}  $ holds for the $ C^0$-norm as $ \sum \epsilon_n < + \infty$. Thus its limit $g$ is continuous. Let us show that $g$ indeed lies in $ \mathcal{F}^{d,k,\alpha}$.
First note that for any $ n \ge 1$ and for any $ 1 \le l \le k$ it holds $ D^l f_{\lambda_n} ( 0 ) = 0 $, thus by Taylor integral formula, it holds: 
\begin{equation} \label{maj dk}
 \| D^l f_{\lambda_n} \|_ \infty \le \|  D^k f_{\lambda_n} \|_ \infty \; .
\end{equation}  
Still by \cref{norm f}, it holds: 
\[  \sum_{n \ge 1  }  \| f_{\lambda_n} \|_{q}  \le \sum_{n \ge 1 } \frac{6}{\pi^2 n^2 } =  1  \; .  \]
Note moreover that as $ D^k f_{\lambda_n} ( 0) = 0$, it holds $ \| D^k f_{\lambda_n} \|_\infty  \le \| f_{\lambda_n} \|_q $ thus by \cref{estime f nq} it follows: 
\[  \sum_{n \ge 1  }  \| f_{\lambda_n} \|_{C^k}  \le \sum_{n \ge 1 }  \| f_{\lambda_n} \|_q  \le  1  \; .  \]
Consequently,  the partial sums of the series lie in $ \mathcal{F}^{d,k,\alpha} $ and so does $g$  as  $ \mathcal{F}^{d,k,\alpha} $ is closed for the $C^0$-norm.
\end{proof}
We will now finish the proof of the case $ \alpha > 0$ of \cref{bilip} using \cref{estime f nq,bound}.
First, by the above \cref{bound}, the following map is well defined:
\[I:  \underline \lambda=(\lambda_n)_{n\ge 1} \in (\Lambda , \delta)   \mapsto \lim_{n\rightarrow  + \infty}  \sum_{n \ge 1  }   f_{\lambda_n} \in  ( \mathcal{F}^{d,k,\alpha}, \| \cdot \|_\infty  )  \; .\]
To conclude the proof, it remains to show \cref{eq exp} of \cref{bilip}.
Consider $ \underline \lambda = ( \lambda_n)_{n \ge 1} $ and $ \underline{\lambda}'= ( \lambda'_n)_{n \ge 1}  \in \Lambda$. Denote $ k :=\nu(\underline \lambda, \underline{\lambda}') $ the first index such that the sequences $ \lambda $ and $ \lambda'$ differ. Then it holds:
\begin{equation} \label{eqi1} 
 \displaystyle \| I(\underline \lambda) - I ( \underline{\lambda}') \|_\infty  = \left\| \sum_{n  \ge k }  f_{\lambda_n } - f_{\lambda'_n } \right\|_\infty \ge  \| f_{\lambda_k} - f_{\lambda'_k} \|_\infty  - \sum_{n >  k} \| f_{\lambda_n} - f_{\lambda'_n} \|_\infty    \; .
\end{equation} 
Now \cref{estime f nq} provides: 
\begin{equation} \label{eqi2} 
 \| f_{\lambda_k} - f_{\lambda'_k} \|_\infty  =  \epsilon_{k} \qand  \sum_{n >  k} \| f_{\lambda_n} - f_{\lambda' _n} \|_\infty  \le  \sum_{n > k}  \epsilon_n   \; . 
\end{equation}
Note that \cref{eps choice} implies: 
\begin{equation}\label{eqi3} 
 \sum_{n > k}  \epsilon_n  \le \sum_{n > k}  \epsilon_k \cdot R^{-q ( n - k )} = \epsilon_k \cdot \frac{1}{R^q - 1} \; . 
\end{equation}
As $ R^q \ge 3$, it holds then $\frac{1}{R^q - 1}  \le \frac{1}{2} $ and consequently combining \cref{eqi1,eqi2,eqi3} leads to: 
   \[ \| I(\underline \lambda)-I(\underline \lambda') \|_\infty 
    \ge \tfrac12 \epsilon_{\nu(\underline \lambda,\underline \lambda')}  \; .\]
Since $ \epsilon_{\nu( \underline \lambda, \underline{\lambda}') } = \delta ( \underline \lambda, \underline{\lambda}' ) $ by definition of  $\delta$, the desired result comes. 
\end{proof} 
It remains to deduce the case $ \alpha = 0$ from that previous one.
For any $ \beta > 0$, it holds $ \mathcal{F}^{d,k,\beta} \subset \mathcal{F}^{d,k,0} $. From there, since Hausdorff scales are non decreasing for inclusion, it holds then $ \ord_H \mathcal{F}^{d,k,0} \ge \ord_H \mathcal{F}^{d,k, \beta } \ge \frac{d}{k + \beta}$. Since we can take $ \beta > 0$ arbitrarily small, it indeed holds $ \ord_H \mathcal{F}^{d,k,0} \ge \frac{d}{k }$.
\end{proof}
\bibliographystyle{alpha}
\bibliography{scales.bib} 
\end{document}